\documentclass[leqno,12pt]{amsart}
\usepackage{amssymb}
\usepackage{amsmath}
\usepackage{enumerate}
\usepackage{amsfonts}
\usepackage{hyperref}
\usepackage{soulutf8}
\usepackage{tikz}

\hypersetup{
    colorlinks=true,
    linkcolor= blue,
    citecolor =cyan,
    urlcolor = teal,
}

\newtheorem{theorem}{Theorem}[section]
\newtheorem{corollary}[theorem]{Corollary}
\newtheorem{lemma}[theorem]{Lemma}
\newtheorem{proposition}[theorem]{Proposition}
\newtheorem{remark}[theorem]{Remark}
\newtheorem{definition}[theorem]{Definition}

\numberwithin{equation}{section}

\newcommand{\supp}{\text{\rm supp}\,}

\newcommand{\Hess}{\operatorname{Hess}}

\begin{document}

\title[Dimension-free $L^p$-estimates for vectors of Riesz transforms] {Dimension-free $L^p$ estimates for vectors of Riesz transforms in the rational Dunkl setting}

\begin{abstract}
    In this article, we prove dimension-free upper bound for the $L^p$-norms of the vector of Riesz transforms in the rational Dunkl setting. Our main technique is Bellman function method adapted to the Dunkl setting.
\end{abstract}


\author[Agnieszka Hejna]{Agnieszka Hejna}

\address{A. Hejna, Uniwersytet Wroc\l awski,
Instytut Matematyczny,
Pl. Grunwaldzki 2/4,
50-384 Wroc\l aw,
Poland}
\email{hejna@math.uni.wroc.pl}

\thanks{
Research supported by the National Science Centre, Poland (Narodowe Centrum Nauki), Grant 2018/31/B/ST1/00204}

\maketitle

\section{Introduction}

In the seminal article~\cite{D2}, Charles F. Dunkl defined new commuting differential-difference operators
\begin{align*}
    T_\xi f(\mathbf x)=\partial_\xi f(\mathbf x)+\sum_{\alpha\in R} \frac{k(\alpha)}{2}\langle \alpha,\xi \rangle  \frac{f(\mathbf x)-f(\sigma_\alpha(\mathbf x))}{\langle \alpha, \mathbf x\rangle}
\end{align*}
associated with a
finite reflection group $G$ which is related to a root system $R$  on a Euclidean space $\mathbb{R}^N$. Here $\xi \in \mathbb{R}^N$, $\sigma_\alpha$  denotes the reflection with respect to the hyperplane orthogonal to the root $\alpha\in R$, and $k:R\to \mathbb C$ is a $G$-invariant function (see Section~\ref{sec:Dunkl} for details). The Dunkl operators are generalizations of the directional derivatives (in fact, they are ordinary partial derivatives for $k \equiv 0$), however, in general, they are non-local operators. They turn out to be a key tool in the study of special functions with reflection symmetries and allow to built up the framework for the theory of special
functions and integral transforms in several variables related with reflection groups in~\cite{D1}--\cite{D5}. Afterwards, the theory was studied and developed by many mathematicians from many different points of view. Beside the special functions and mathematical analysis, the Dunkl theory has deep connections with the other branches of mathematics, for instance probability theory, mathematical physics, and algebra.

The aim of the article is to study the Riesz transforms in the rational Dunkl setting defined as follows.

\begin{definition}\normalfont\index{Riesz transforms}\index{R@$R_{j}$}
Let $f \in \mathcal{S}(\mathbb{R}^N)$ and $j \in \{1,\ldots,N\}$. The \textit{Riesz transforms} $R_{j}$ in the Dunkl setting are defined by
\begin{equation}\label{eq:Riesz_mult}
    \mathcal{F} (R_jf)(\xi) = - i\frac{\xi_j}{\|\xi\|} (\mathcal{F}f)(\xi),
\end{equation}
where $\mathcal{F}$ is the Dunkl transform (see~\eqref{eq:transform}). The \textit{vector of the Riesz transforms} in the Dunkl setting is defined by
\begin{equation}
    \mathcal{R}f(\mathbf{x})=\left(\sum_{j=1}^{N}|R_jf(\mathbf{x})|^2\right)^{1/2}.
\end{equation}
Here and subsequently, $\mathcal{S}(\mathbb{R}^N)$ denotes the Schwartz class functions.
\end{definition}

The Riesz transforms in the Dunkl setting were introduced in~\cite[Theorem 5.3]{ThangaveluXu_Riesz}. The following theorem was proved in~\cite[Theorem 3.3]{AmriSifi_Riesz}. 

\begin{theorem}[{\cite[Theorem 3.3]{AmriSifi_Riesz}}]\label{teo:Amri}
Let $1<p<\infty$. The Riesz transforms, defined initially on $\mathcal{S}(\mathbb{R}^N)$, extend to bounded operators $L^p(dw) \longmapsto L^p(dw)$, where $dw$ is the measure associated with the root system $R$ and the multiplicity function $k$ (see~\eqref{eq:measure} and Section~\ref{sec:Dunkl} for details).
\end{theorem}

Moreover, it can be checked by using the Dunkl transform (see Lemma~\ref{lem:transform_properties}) that 
\begin{equation}\label{eq:Riesz_formula}
    R_jf=-T_{j} (-{\Delta_{k}})^{-1\slash 2} f
\end{equation}
for $f\in L^2({dw})$, where $\Delta_k=\sum_{j=1}^N T_{e_j}^2$ is the Dunkl Laplacian. Here and subsequently, $\{e_j\}_{1 \leq j \leq N}$ denote the canonical orthonormal basis in $\mathbb R^N$.

A well-known result concerning the classical Riesz transforms, proved by E.M. Stein in~\cite{Stein}, stated that in the case $k \equiv 0$, there are upper bounds for the $L^p$-norm of the vector of the Riesz transforms independent of the dimension $N$. Then it was proved that, in fact, the $L^p$-norm of the vector of the Riesz transforms is controlled by $C\max(p,\frac{p}{p-1})$, where $C>0$ is independent of $p$ and the dimension $N$, see~\cite{Ba,Duo_Fra}. At this point, it is worth to mention that in the case $k \equiv 0$, the norms of the vector of the Riesz transforms are still not known (see~\cite{Ba_1,DV1,Iwaniec} for the some results concerning the subject).

The aim of the current paper is to prove the bounds for $L^p(dw)$-norms of the vector of Riesz transforms in that spirit in the rational Dunkl setting, i.e., the case of $k \not \equiv 0$ and for any root system $R$.

The main goal of this paper is to prove the following theorem. Recall that the measurable function $ $ is \textit{$G$-invariant}, if for almost all $\mathbf{x} \in \mathbb{R}^N$ and $\sigma \in G$ we have 
\begin{align*}
    f(\sigma(\mathbf{x}))=f(\mathbf{x})
\end{align*}
(see Section~\ref{sec:Dunkl} for the definition of the Weyl group $G$).

\begin{theorem}\label{teo:main}
Let $p,q>1$ be such that $\frac{1}{p}+\frac{1}{q}=1$. Set $p^{*}=\max(p,q)$. Then for all $f \in L^p(dw)$ we have
\begin{equation}\label{eq:m_1}
    \|\mathcal{R}f\|_{L^p(dw)} \leq 144(p^{*}-1)\left(\sum_{\alpha \in R}k(\alpha)+2^{7}\right)\|f\|_{L^p(dw)}.
\end{equation}
Moreover, for all $f \in L^p(dw)$, which are $G$-invariant, we have
\begin{equation}\label{eq:m_2}
    \|\mathcal{R}f\|_{L^p(dw)} \leq 144(p^{*}-1)\|f\|_{L^p(dw)}.
\end{equation}
\end{theorem}
Our second main goal in this paper will be to prove a different version of Theorem~\ref{teo:main} in the one-dimensional case. If $N=1$, then there is just one Riesz transform (Dunkl Hilbert transform), which will be denoted by $\mathcal{H}$, i.e.,
\begin{equation}\label{eq:Hilbert}
    \mathcal{F} (\mathcal{H}f)(\xi) = - i\frac{\xi}{|\xi|} (\mathcal{F}f)(\xi), \ \ \xi \in \mathbb{R}^N.
\end{equation}

\begin{theorem}\label{teo:main2}
Assume that $N=1$. Let $p,q>1$ be such that $\frac{1}{p}+\frac{1}{q}=1$. Set $p^{*}=\max(p,q)$. Then for all $f \in L^p(dw)$ we have
\begin{equation}
    \|\mathcal{H}f\|_{L^p(dw)} \leq 1440(p^{*}-1)\|f\|_{L^p(dw)}.
\end{equation}
\end{theorem}

The dimension free estimates for vector of Riesz transforms has been studied by many authors. For instance, the estimates of that spirit for the Riesz transforms and the vector of the Riesz transforms we considered in the following contexts:
\begin{itemize}
    \item{Ornstein-Uhlenbeck operator (see~\cite{OU_5,CD1,OU_6,DV1,OU_2,OU_3,OU_7,OU_1,OU_4});}
    \item{Laguerre operator (see~\cite{L_3,L_1,Mauceri,L_2,L_5,L_4});}
    \item{Jacobi operator (see~\cite{J_1,J_2,J_3,J_4});}
    \item{Harmonic oscillator (see~\cite{H_1,Sch,H_2,H_4,H_3});}
    \item{Bessel operator (see~\cite{Bessel});}
    \item{Grushin operator (see~\cite{Grushin});}
    \item{General context of orthogonal expansions (see~\cite{G_1,Wrobel});}
    \item{Weighted Riesz transforms (see~\cite{Dome});}
    \item{Noncommutative Riesz transforms (see~\cite{Non}).}
\end{itemize}

The main tool that is used in the current paper is the Bellman function method (see Section~\ref{sec:Bellman}).  This  method  was  introduced  by  Nazarov,  Treil  and  Volberg  in~\cite{Nazarov}. Bellman  functions in implicit forms were previously used by Burkholder in~\cite{Burk_1,Burk_2,Burk_3}. Then, the approach based on careful studying the properties of the Bellmann function was developed by Dragicevic and Volberg in the series of papers~\cite{DV1,DV2,Sch}, and then by Carbonaro and Dragicevic in~\cite{CD1,CD2,CD3,CD4}.

Let us discuss some difficulties in Dunkl analysis, which distinguish it from the classical setting $k \equiv 0$. As it was pointed out in~\cite{ThangaveluXu}, one of the most serious problem in the Dunkl analysis lays in the lack of knowledge about generalized translations $\tau_{\mathbf{x}}$, $\mathbf{x} \in \mathbb{R}^N$, which generalize the ordinary translation of the function $f \longmapsto f(\cdot-\mathbf{x})$. It was proved that for some root systems $R$ the operators $\tau_{\mathbf{x}}$ do not preserve positive functions and the boundedness of $\tau_{\mathbf{x}}$ on $L^p(dw)$-spaces ($p\ne 2$) becomes an open problem in the Dunkl analysis. In the context of this paper, we overcome this difficulty using recently proved upper bounds for the Dunkl Poisson kernel (see~\cite{ADzH}). 

Looking from the point of view of the current paper, let us discuss another difficulty regarding the Dunkl operators. The Dunkl operators $T_\xi$ do not satisfy the Leibniz rule in the usual sense, i.e., the formula
\begin{align*}
    T_\xi(fg)(\mathbf{x})=f(\mathbf{x})T_{\xi}g(\mathbf{x})+T_{\xi}f(\mathbf{x})g(\mathbf{x})
\end{align*}
holds just in specific cases e.g. if $f$ or $g$ is radial. In general case, the formula for $T_\xi(fg)$ contains summands of local and non-local character. The analysis turns to be more complicated when we compose two or more Dunkl operators, which is the case when we are trying to adapt the Bellman function method.

At this point, it is also worth to mention that in the Dunkl setting the explicit formulas for $\Delta_k u^p$ for $p \in [1,\infty)$ and $u \in \mathcal{S}(\mathbb{R}^N)$ seem to be of quite different nature than in the case $k \equiv 0$. In order to elaborate the case of $p=2$, let us consider the Dunkl version of the carr\'e du champ operator: 
\begin{equation*}\label{form_Gamma} \Gamma_k (f,g)=\frac{1}{2}\Big(\Delta_k (f  g)-f\Delta_k  g- g\Delta_k f\Big).
\end{equation*}
As it was noticed in~\cite{Velicu1}, we have
\begin{equation*}
    \int_{\mathbb{R}^N} \Gamma_k (f,g)\,dw=\int_{\mathbb{R}^N}\sum_{j=1}^{N}T_jfT_jg\,dw,
\end{equation*}
but the identity $\sum_{j=1}^{N}T_jfT_jg \equiv \Gamma_k (f,g)$ is not true if $k \not \equiv 0$, which can be checked by the explicit calculation:
\begin{equation*}\label{form_Gamma1} \Gamma_k(f,g)(\mathbf x)=\langle \nabla f(\mathbf x),\nabla g(\mathbf x)\rangle + \sum_{\alpha \in R}\frac{k(\alpha)}{2}\frac{(f(\mathbf x)-f(\sigma_\alpha(\mathbf x)))({ g(\mathbf x)-g(\sigma_\alpha(\mathbf x)}))}{\langle \alpha,\mathbf x\rangle^2}
\end{equation*} 
(see also~\cite{Graczyk} for a more general calculation). In the current paper, following the approach presented in~\cite{Sch}, we obtain an explicit formula for $\Delta_k$ applied to the Bellman function, which turn out to be closely related to the known formulas for $\Delta_k u^p$. Therefore, the Bellman approach has to be adapted to this specific setting.

{\bf Acknowledgment.} The author would like to thank Błażej Wróbel and Jacek Dziuba\'nski for their helpful comments and suggestions, and Charles Dunkl for pointing our some references.

\section{Basic definitions of the Dunkl theory}\label{sec:Dunkl}

In this section, for the convenience of the reader, we present basic facts concerning the theory of the Dunkl operators.  For details we refer the reader to~\cite{D2},~\cite{Roesler_notes}, and~\cite{RoeslerVoit}. The reader who is familiar with the Dunkl theory can omit this section and proceed to Subsection~\ref{sec:ess}. 

We consider the Euclidean space $\mathbb R^N$ with the scalar product $\langle\mathbf x,\mathbf y\rangle=\sum_{j=1}^N x_jy_j
$, where $\mathbf x=(x_1,...,x_N)$, $\mathbf y=(y_1,...,y_N)$, and the norm $\| \mathbf x\|^2=\langle \mathbf x,\mathbf x\rangle$. The number $N$ will be fixed throughout this paper. For a nonzero vector $\alpha\in\mathbb R^N$,  the reflection $\sigma_\alpha$ with respect to the hyperplane $\alpha^\perp$ orthogonal to $\alpha$ is given by
\begin{equation}\label{eq:x_minus_sigma_x}
\sigma_\alpha (\mathbf x)=\mathbf x-2\frac{\langle \mathbf x,\alpha\rangle}{\| \alpha\| ^2}\alpha.
\end{equation}
In this paper we fix a normalized root system in $\mathbb R^N$, that is, a finite set  $R\subset \mathbb R^N\setminus\{0\}$ such that $R \cap \alpha \mathbb{R} = \{\pm \alpha\}$,  $\sigma_\alpha (R)=R$, and $\|\alpha\|=\sqrt{2}$ for all $\alpha\in R$. The finite group $G$ generated by the reflections $\sigma_\alpha$, $\alpha \in R$ is called the {\it Weyl group} ({\it reflection group}) of the root system. A~{\textit{multiplicity function}} is a $G$-invariant function $k:R\to\mathbb C$ which will be $\geq 0$  throughout this paper. 
 Let
\begin{equation}\label{eq:measure}
dw(\mathbf x)=\prod_{\alpha\in R}|\langle \mathbf x,\alpha\rangle|^{k(\alpha)}\, d\mathbf x
\end{equation} 
be  the associated measure in $\mathbb R^N$, where, here and subsequently, $d\mathbf x$ stands for the Lebesgue measure in $\mathbb R^N$. For a Lebesgue measurable set $A$ we denote $w(A)=\int_{A}\,dw(\mathbf{x})$.
There is a constant $C>0$ such that 
\begin{equation}\label{eq:balls_asymp} 
C^{-1}w(B(\mathbf x,r))\leq  r^{N}\prod_{\alpha \in R} (|\langle \mathbf x,\alpha\rangle |+r)^{k(\alpha)}\leq C w(B(\mathbf x,r)),
\end{equation}
so $dw(\mathbf x)$ is doubling, that is, there is a constant $C>0$ such that
\begin{equation}\label{eq:doubling} w(B(\mathbf x,2r))\leq C w(B(\mathbf x,r)) \ \ \text{ for all } \mathbf x\in\mathbb R^N, \ r>0.
\end{equation}
Moreover, since the function $w$ is $G$-invariant, for all $\sigma \in G$ we have
\begin{equation}\label{eq:integral_G_invariance}
    \int_{\mathbb{R}^N}f(\sigma(\mathbf{x}))\,dw(\mathbf{x})=\int_{\mathbb{R}^N}f(\mathbf{x})\,dw(\mathbf{x}).
\end{equation}

For $\xi \in \mathbb{R}^N$, the {\it Dunkl operators} $T_\xi$  are the following $k$-deformations of the directional derivatives $\partial_\xi$ by a  difference operator:
\begin{equation}\label{eq:Dunkl_op}
     T_\xi f(\mathbf x)= \partial_\xi f(\mathbf x) + \sum_{\alpha\in R} \frac{k(\alpha)}{2}\langle\alpha ,\xi\rangle\frac{f(\mathbf x)-f(\sigma_\alpha(\mathbf{x}))}{\langle \alpha,\mathbf x\rangle}.
\end{equation}
The Dunkl operators $T_{\xi}$, which were introduced in~\cite{D2}, commute and are skew-symmetric with respect to the $G$-invariant measure $dw$.
\
Let $\{e_j\}_{1 \leq j \leq N}$ denote the canonical orthonormal basis in $\mathbb R^N$ and let $T_j=T_{e_j}$. As usual, for every multi-index \hspace{.5mm}$\beta\hspace{-.5mm}=\hspace{-.5mm}(\beta_1,\beta_2,\dots,\beta_N)\!\in\hspace{-.5mm}\mathbb{N}_0^N=(\mathbb{N} \cup \{0\})^{N}$,
we set \hspace{.25mm}$|\beta|\!
=\hspace{-.5mm}\sum_{\hspace{.25mm}j=1}^{\hspace{.5mm}N}\hspace{-.25mm}\beta_j$
and
$$
\partial^{\hspace{.25mm}\beta}\!
=\partial_{e_1}^{\hspace{.25mm}\beta_1}\!\circ\hspace{-.25mm}
\partial_{e_2}^{\hspace{.25mm}\beta_{\hspace{.2mm}2}}
\!\circ\ldots\circ\hspace{-.25mm}
\partial_{e_{N}}^{\hspace{.25mm}\beta_{\hspace{-.2mm}N}}\,,
$$
where $\{e_1,e_2\hspace{.25mm},\ldots,e_{N}\}$ is the canonical basis of $\mathbb{R}^N$.
The additional subscript $\mathbf{x}$ in $\partial^{\hspace{.25mm}\alpha}_{\mathbf{x}}$
means that the partial derivative \hspace{.25mm}$\partial^{\hspace{.25mm}\alpha}$
is taken with respect to the variable $\mathbf{x}\!\in\!\mathbb{R}^N$. By $\nabla_{\mathbf{x}}f$ we denote the gradient of the function $f$ with respect to the variable $\mathbf{x}$. 
\

The following fundamental theorem was proved by Ch. Dunkl.

\begin{theorem}[\cite{D5}]\label{teo:by_parts}
The Dunkl operators are skew-symmetric with respect to the measure $dw$. More precisely, for any $\xi \in \mathbb{R}^N$, $f \in \mathcal{S}(\mathbb{R}^{N})$, and $g \in C^1_{b}(\mathbb{R}^N)$\index{C@$C^1_b(\mathbb{R}^N)$} (here and subsequently, $C_b^1(\mathbb{R}^N)$ denotes the set of bounded functions with bounded and continuous partial derivatives), we have the following integration by parts formula
\begin{equation}\label{eq:by_parts}
    \int_{\mathbb{R}^N}T_{\xi}f(\mathbf{x})g(\mathbf{x})\,dw(\mathbf{x})=-\int_{\mathbb{R}^N}f(\mathbf{x})T_{\xi}g(\mathbf{x})\,dw(\mathbf{x}).
\end{equation}
\end{theorem}

\begin{remark}\label{rem:by_parts}\normalfont
Note that~\eqref{eq:by_parts} holds also if $f \in C^{\infty}_c(\mathbb{R}^N)$ and $g \in C^{1}(\mathbb{R}^N)$. In order to justify this fact, it is enough to take $\varphi \in C^{\infty}_c(\mathbb{R}^N)$ such that $\mathbf{\varphi} \equiv 1$ on
\begin{align*}
   A=\bigcup_{\sigma \in G}\sigma(\supp f).
\end{align*}
It follows from~\eqref{eq:Dunkl_op} that $T_\xi f \equiv 0$ and $f \equiv 0$ on $\mathbb{R}^N \setminus A$. Hence, by~\eqref{eq:by_parts} we have
\begin{align*}
    &\int_{\mathbb{R}^N}T_{\xi}f(\mathbf{x})g(\mathbf{x})\,dw(\mathbf{x})=\int_{\mathbb{R}^N}T_{\xi}f(\mathbf{x})(\mathbf{\varphi}g)(\mathbf{x})\,dw(\mathbf{x})\\&=-\int_{\mathbb{R}^N}f(\mathbf{x})T_{\xi}(\mathbf{\varphi}g)(\mathbf{x})\,dw(\mathbf{x})=-\int_{\mathbb{R}^N}f(\mathbf{x})T_{\xi}g(\mathbf{x})\,dw(\mathbf{x}).
\end{align*}

\end{remark}
We will also need the following technical lemma, which is well-known. We provide the sketch of its proof for the sake of completeness.

\begin{lemma}\label{lem:bdd}
For any $\beta \in \mathbb{N}_0^{N}$ there is a constant $C_{\beta}>0$ such that for all $f \in C^{\infty}(\mathbb{R}^N)$ we have
\begin{align*}
    \|T^{\beta}f\|_{L^{\infty}} \leq C_{\beta}\sum_{\beta' \in \mathbb{N}_0^{N},\, |\beta'| = |\beta|}\|\partial^{\beta'}f\|_{L^{\infty}}.
\end{align*}
\end{lemma}

\begin{proof}
By the definition of $T_j$ {and by the fundamental theorem of calculus}, for all $f \in C^{1}(\mathbb{R}^N)$, we have
\begin{equation*}\begin{split} T_jf(\mathbf x)&
=\partial_j f(\mathbf x)-\sum_{\alpha\in R} \frac{k(\alpha)}{2}\alpha_j\langle \mathbf{x}, \alpha \rangle^{-1}\int_0^{1} \frac{d}{dt}(\phi(\mathbf{x}-2t\alpha \|\alpha\|^{-2}\langle \mathbf{x}, \alpha\rangle ))\,dt\\
&=\partial_j f(\mathbf x)+ \sum_{\alpha\in R}\frac{k(\alpha)}{2}\alpha_j \int_0^1 \langle {(}\nabla_{\mathbf{x}} f{)}(\mathbf{x}-2t\alpha \|\alpha\|^{-2}\langle \mathbf{x}, \alpha\rangle ), \alpha \rangle \,dt
\end{split}\end{equation*}
(cf.~\cite[page 9]{RoeslerVoit}). Consequently, for any $\beta \in \mathbb{N}_0^N$ there is a constant $C>0$ such that for all $f \in C^{|\beta|+1}(\mathbb{R}^N)$  and $j \in \{1,\ldots,N\}$ we have
\begin{equation}\label{eq:s_x_1}
\sup_{\mathbf x\in\mathbb R^N} |\partial^\beta T_j f(\mathbf x)|\leq C \sup_{\mathbf x\in\mathbb R^N} \|\nabla_{\mathbf{x}}\partial^{\beta}f(\mathbf x)\|.
\end{equation}
The claim follows from~\eqref{eq:s_x_1} by the induction on $|\beta|$.
\end{proof}

For fixed $\mathbf y\in\mathbb R^N$ the {\it Dunkl kernel} $E(\mathbf x,\mathbf y)$ is the unique analytic solution to the system
\begin{align*}
    T_\xi f=\langle \xi,\mathbf y\rangle f, \ \ f(0)=1.
\end{align*}
The function $E(\mathbf x ,\mathbf y)$, which generalizes the exponential  function $e^{\langle \mathbf x,\mathbf y\rangle}$, has the unique extension to a holomorphic function on $\mathbb C^N\times \mathbb C^N$.

The \textit{Dunkl transform} is defined by
  \begin{equation}\label{eq:transform}\mathcal F f(\xi)=c_k^{-1}\int_{\mathbb R^N} E(-i\xi, \mathbf x)f(\mathbf x)\, dw(\mathbf x),
  \end{equation}
  where
  $$c_k=\int_{\mathbb{R}^N}e^{-\frac{\|\mathbf{x}\|^2}{2}}\,dw(\mathbf{x})>0,$$
  for $f\in L^1(dw)$. It was introduced in~\cite{D5} for $k \geq 0$ and further studied in~\cite{dJ1} in the more general context. It was proved in~\cite[Corollary 2.7]{D5} (see also~\cite[Theorem 4.26]{dJ1}) that is an isometry on $L^2(dw)$, i.e.,
   \begin{equation}\label{eq:Plancherel}
       \|f\|_{L^2(dw)}=\|\mathcal{F}f\|_{L^2(dw)} \text{ for all }f \in L^2(dw).
   \end{equation}
We have also the following inversion theorem.
\begin{theorem}[Inversion theorem, see{~\cite[Theorem 4.20]{dJ1}}]\label{teo:inversion}
For all $f \in L^1(dw)$ such that $\mathcal{F}f \in L^1(dw)$ we have
\begin{equation}\label{eq:inverse}
    f(\mathbf{x})=(\mathcal{F})^2f(-\mathbf{x}) \text{ for almost all }\mathbf{x} \in \mathbb{R}^N.
\end{equation}
The inverse $\mathcal F^{-1}$ of $\mathcal{F}$\index{F@$\mathcal{F}^{-1}$} has the form
  \begin{equation}\label{eq:inverse_teo} \mathcal F^{-1} f(\mathbf{x})=c_k^{-1}\int_{\mathbb R^N} f(\xi)E(i\xi, \mathbf x)\, dw(\xi)= \mathcal{F}f(-\mathbf{x}).
  \end{equation}
\end{theorem}
Below we list some properties of $\mathcal{F}$. 
\begin{lemma}\label{lem:transform_properties}
Suppose that $f \in \mathcal{S}(\mathbb{R}^N)$ and $j \in \{1,\ldots,N\}$. Then we have
\begin{enumerate}[(A)]
    \item{$\mathcal{F}f \in \mathcal{S}(\mathbb{R}^N)$;}\label{numitem:Schwartz_inv}
    \item{ $T_{j}(\mathcal{F}f)(\xi)=\mathcal{F}g(\xi)$, where $g(\xi)=-i\xi_{j}f(\xi)$;}\label{numitem:transform_on_der_2}
    \item{ $\mathcal{F}(T_{j}f)(\xi)=i\xi_{j}\mathcal{F}f(\xi)$.}\label{numitem:transform_on_der}
\end{enumerate}
\end{lemma}
\subsection{Dunkl Laplacian}

\begin{definition}\normalfont
The \textit{Dunkl Laplacian}\index{Dunkl Laplacian} associated with $G$ and $k$  is the differential-difference operator
\begin{equation}\label{eq:laplacian}
    \Delta_k=\sum_{j=1}^N T_{j}^2.    
\end{equation}
It was introduced in~\cite{D2}, where it was also proved that $\Delta_k$ acts on $C^2(\mathbb{R}^N)$ functions by
\begin{equation}\label{eq:laplace_formula}
    \Delta_k f(\mathbf x)=\Delta f(\mathbf x)+\sum_{\alpha\in R} k(\alpha) \delta_\alpha f(\mathbf x),
\end{equation}
\begin{equation*}
    \delta_\alpha f(\mathbf x)=\frac{\partial_\alpha f(\mathbf x)}{\langle \alpha , \mathbf x\rangle} -  \frac{f(\mathbf x)-f(\sigma_\alpha (\mathbf x))}{\langle \alpha, \mathbf x\rangle^2}.
\end{equation*}
Here and subsequently, $\Delta=\sum_{j=1}^{N}\partial_j^2$.
\end{definition}

We have the following theorem, which allows us to define $\sqrt{-\Delta_k}$ by spectral theorem.
\begin{theorem}[{\cite[Theorem 4.8]{R1998}}]\label{teo:Laplace_closed}
The operator $(-\Delta_{k},\mathcal{S}(\mathbb{R}^N))$ in $L^2(dw)$ is densily defined and closable. Its closure will be denoted by the same symbol $-\Delta_k$, is self-adjoint and its domain is
\begin{equation*}
    \mathcal{D}(-\Delta_k)=\{f \in L^2(dw)\;:\; \|\xi\|^2(\mathcal{F}f)(\xi) \in L^2(dw(\xi))\}.
\end{equation*}
It is the unique positive self-adjoint
extension of $(-\Delta_k,\mathcal{S}(\mathbb{R}^N))$.
\end{theorem}

Note that, thanks to Lemma~\ref{lem:transform_properties}~\eqref{numitem:transform_on_der}, for all $\xi \in \mathbb{R}^N$ and $f \in \mathcal{S}(\mathbb{R}^N)$ we have
\begin{equation}\label{eq:Laplacian_on_Fourier_side}
    \mathcal{F}(\Delta_{k}f)(\xi)=-\|\xi\|^2\mathcal{F}f(\xi),
\end{equation}
therefore
\begin{equation}\label{eq:half_Laplacian_on_Fourier_side}
    \mathcal{F}((-\Delta_{k})^{1/2}f)(\xi)=-\|\xi\|\mathcal{F}f(\xi).
\end{equation}

\section{Dunkl Poisson semigroup and \texorpdfstring{$L^p(dw)$}{Lp}-norm of Riesz transform in terms of integral involing Dunkl Poisson semigroup}
\subsection{\texorpdfstring{$k$}{k}-Cauchy kernel and Dunkl Poisson semigroup}

\begin{definition}\label{def:Poisson_semigroup}\normalfont\index{k@$k$-Cauchy kernel}\index{generalized Poisson kernel}\index{p@$p_{t}(\mathbf{x},\mathbf{y})$}\index{P@$P_t$}\index{Poisson semigroup}
Let $\mathbf{x},\mathbf{y} \in \mathbb{R}^N$ and $t>0$. We define the \textit{$k$-Cauchy kernel} $p_t(\mathbf x,\mathbf y)$ to be the integral kernel of the operator $P_t=e^{-t\sqrt{-\Delta_{k}}}$ on $L^2(dw)$ (see Theorem~\ref{teo:Laplace_closed}), that is
\begin{align*}
    P_tf(\mathbf{x})=\int_{\mathbb{R}^N}p_t(\mathbf{x},\mathbf{y})f(\mathbf{y})\,dw(\mathbf{y}).
\end{align*}
\end{definition}

The kernel $p_t(\mathbf x,\mathbf y)$ was introduced and
studied in~\cite{RoeslerVoit}.

\begin{theorem}[{\cite[Theorem 5.6]{RoeslerVoit}}]\label{teo:Markov}
Let $f$ be a bounded continuous function on $\mathbb{R}^N$. Then the function given by $v(\mathbf{x},t)=P_tf(\mathbf{x})$ is continuous and bounded. Moreover, it solves the Cauchy problem
\begin{equation*}
\begin{cases}
    \partial_{t}^2v(\mathbf{x},t)+\Delta_{k,\mathbf{x}}v(\mathbf{x},t)=0 \text{ on }\mathbb{R}^N \times (0,\infty),\\
    v(\mathbf{x},0)=f(\mathbf{x}) \text{ for all }\mathbf{x} \in \mathbb{R}^N.
\end{cases}
\end{equation*}
\end{theorem}
The $k$-Cauchy kernel is also called the \textit{generalized Poisson kernel} (or \textit{Dunkl Poisson kernel}) by the analogy with the classical Poisson semigroup. We have the following lemma.

\begin{lemma}\label{lem:poisson_properties}
Let $\mathbf{x},\mathbf{y} \in \mathbb{R}^N$ and $t>0$. The generalized Poisson kernel $p_t(\mathbf{x},\mathbf{y})$ has the following properties:
\begin{enumerate}[(A)]
    \item{$p_t(\mathbf{x},\mathbf{y})=p_t(\mathbf{y},\mathbf{x})$;}\label{numitem:poisson_symmetric}
    \item{$\int_{\mathbb{R}^N}p_{t}(\mathbf{x},\mathbf{z})\,dw(\mathbf{z})=1$;}\label{numitem:poisson_integral_one}
    \item{$p_t(\mathbf{x},\mathbf{y})>0$;}\label{numitem:poisson_positive}
    \item{$p_t(\sigma(\mathbf{x}),\sigma(\mathbf{y}))=p_t(\mathbf{x},\mathbf{y})$ for all $\sigma \in G$.}\label{numitem:G_invariant}
\end{enumerate}
\end{lemma}

It follows by Theorem~\ref{teo:Markov},~\eqref{eq:half_Laplacian_on_Fourier_side}, and the inversion theorem for Dunkl transform (see Theorem~\ref{teo:inversion}) that for all $f \in \mathcal{S}(\mathbb{R}^N)$, $\mathbf{x} \in \mathbb{R}^N$, and $t>0$ we have
\begin{equation}\label{eq:Poisson_transform_form}
    P_tf(\mathbf{x})=c_k^{-1} \int_{\mathbb{R}^{N}}e^{-t\|\xi\|}E(i\xi,\mathbf{x})\mathcal{F}f(\xi)\,dw(\xi).
\end{equation}
We also have the following upper and lower bound for the generalized Poisson kernel.
\begin{proposition}[{\cite[Proposition 5.1]{ADzH}}]\label{Poiss_new}
For $\mathbf{x},\mathbf{y} \in \mathbb{R}^N$ and $t,r>0$ we denote
\begin{equation*}\index{V@$V(\mathbf{x},\mathbf{y},r)$}
    V(\mathbf{x},\mathbf{y},r)=\max\{w(B(\mathbf{x},r)),w(B(\mathbf{y},r))\},
\end{equation*}
\begin{equation}\label{eq:distance_of_orbits}
    d(\mathbf x,\mathbf y)=\min_{\sigma\in G}\| \sigma(\mathbf x)-\mathbf y\| .   
\end{equation}
{\rm (a) Upper and lower bounds\,:}
there is a constant $C \geq 1$ such that
\begin{equation}\label{Poisson_low_up}
\frac{C^{-1}}{V(\mathbf{x},\mathbf{y},t+{\|}\mathbf{x}-\mathbf{y}{\|})}\,\frac{t}{t+{\|}\mathbf{x}-\mathbf{y}{\|}}
\leq {p}_t(\mathbf{x},\mathbf{y})\leq\frac{C}{V(\mathbf{x},\mathbf{y},t+d(\mathbf{x},\mathbf{y}))}\,\frac{t}{t+d(\mathbf{x},\mathbf{y})}
\end{equation}
{for all \,$t>0$ and for all \,$\mathbf{x},\mathbf{y}\in\mathbb{R}^N$.}
\par\noindent
{\rm (b) Dunkl gradient\,:}
for every \,$\xi\in\mathbb{R}^N$, there is a constant $C>0$ such that
\begin{equation}\label{TxiPoisson}
\bigl|T_{\xi,\mathbf{y}} {p}_t(\mathbf{x},\mathbf{y})\bigr|\leq
\frac{C}{V(\mathbf{x}, \mathbf{y},t+d(\mathbf{x},\mathbf{y}))}\,\frac{1}{t+d(\mathbf{x},\mathbf{y})}
\end{equation}
{for all \,$t>0$ and for all \,$\mathbf{x},\mathbf{y}\in\mathbb{R}^N$.}
\par\noindent
{\rm (c) Mixed derivatives\,:}
for any nonnegative integer \,$m$ and for any multi-index \,$\beta \in \mathbb{N}_0^{N}$, there is a constant \,$C\hspace{-.5mm}\ge\hspace{-.5mm}0$ such that, for all \,$t>0$ and for all \,$\mathbf{x},\mathbf{y}\in\mathbb{R}^N$,
\begin{equation}\label{DtDyPoisson}
\bigl|\hspace{.25mm}\partial_t^m\partial_{\mathbf{y}}^{\beta}\hspace{.25mm}p_t(\mathbf{x},\mathbf{y})\bigr|\le C\,p_t(\mathbf{x},\mathbf{y})\hspace{.25mm}\bigl(\hspace{.25mm}t\hspace{-.25mm}+d(\mathbf{x},\mathbf{y})\bigr)^{\hspace{-.5mm}-m-|\beta|}\times\begin{cases}
\,1&\text{if \,}m\hspace{-.25mm}=\hspace{-.25mm}0\hspace{.25mm},\\
\,1+\frac{d(\mathbf{x},\mathbf{y})}t&\text{if \,}m\hspace{-.5mm}>\hspace{-.5mm}0\hspace{.25mm}.\\
\end{cases}\end{equation}
Moreover, for any nonnegative integer \,$m$ and for any multi-indices \,$\beta,\beta' \in \mathbb{N}_0^{N}$, there is a constant \,$C\hspace{-.5mm}\ge\hspace{-.5mm}0$ such that, for all \,$t>0$ and for all \,$\mathbf{x},\mathbf{y}\in\mathbb{R}^N$,
\begin{equation}\label{DtDxDyPoisson}
\bigl|\hspace{.25mm}\partial_t^m\partial_{\mathbf{x}}^{\beta}\partial_{\mathbf{y}}^{\beta'}p_t(\mathbf{x},\mathbf{y})\bigr|\le C\,t^{-m-|\beta|-|\beta'|}\,p_t(\mathbf{x},\mathbf{y})\,.
\end{equation}
\end{proposition}

Note that the estimates in Proposition~\ref{Poiss_new} are given in the spirit of spaces of homogeneous type, except that the metric $\|\mathbf{x}-\mathbf{y}\|$ is replaced by the distance of the orbits $d(\mathbf{x},\mathbf{y})$ (see~\eqref{eq:distance_of_orbits}). One of the reason why the estimates of Proposition~\ref{Poiss_new} are suitable in many context is explained in the next lemma. We omit its standard proof.

\begin{lemma}\label{lem:basic}
Let $1 \leq p \leq \infty$. If $f \in L^p(dw)$, then $(\mathbf{x},t) \longmapsto P_tf(\mathbf{x})$ belongs to $C^{\infty}(\mathbb{R}^N \times (0,\infty))$ and for all $(m,\beta) \in \mathbb{N}_0 \times \mathbb{N}_0^{N}$ we have
\begin{align*}
    \partial_t^m\partial_{\mathbf{x}}^{\beta}P_tf(\mathbf{x})=\int_{\mathbb{R}^N}\partial_t^m\partial_{\mathbf{x}}^{\beta}p_t(\mathbf{x},\mathbf{y})f(\mathbf{y})\,dw(\mathbf{y}).
\end{align*}
Moreover, for any $m \in \mathbb{N}_{0}$ there is a constant $C=C_{p,m}>0$ such that for all $t>0$ and $f \in L^p(dw)$ we have
\begin{equation}\label{eq:basic}
    \|\partial_{t}^{m}P_tf\|_{L^p(dw)} \leq C t^{-m}\|f\|_{L^p(dw)}.
\end{equation}
\end{lemma}
\subsection{\texorpdfstring{$L^p(dw)$}{Lp}-norm of Riesz transforms in term of integral involving Dunkl Poisson semigroup}\label{sec:ess}

The next proposition is well-known (see~\cite{Stein},~\cite[Lemma 2.1]{DV1}). We provide its version in the Dunkl setting for the sake of completeness.

\begin{proposition}\label{propo:dual_1}
For all $j \in \{1,\ldots,N\}$ and $f,g \in \mathcal{S}(\mathbb{R}^N)$ we have
\begin{equation}
    \left|\int_{\mathbb{R}^N}R_jf(\mathbf{x})g(\mathbf{x})\,dw(\mathbf{x})\right|=4\left|\int_{\mathbb{R}^N}\int_0^{\infty}t\partial_tP_tg(\mathbf{x})T_jP_tf(\mathbf{x})\,dt\,dw(\mathbf{x})\right|.
\end{equation}
\end{proposition}

\begin{proof}
For $1 \leq j \leq N$, $\mathbf{x} \in \mathbb{R}^N$, and $t>0$ we define
\begin{align*}
    \varphi(\mathbf{x},t):=P_tR_jf(\mathbf{x})P_tg(\mathbf{x}).
\end{align*}
It follows by Proposition~\ref{Poiss_new} that for fixed $\mathbf{x} \in \mathbb{R}^N$ there is a constant $C>0$ independent of $\mathbf{x}$ such that for all $\mathbf{y} \in \mathbb{R}^N$ and $t>0$ we have
\begin{align*}
    p_t(\mathbf{x},\mathbf{y}) \leq \frac{C}{w(B(\mathbf{x},t))}.
\end{align*}
Hence, for all $F \in \mathcal{S}(\mathbb{R}^N)$ we have
\begin{align*}
    |P_tF(\mathbf{x})| \leq \int_{\mathbb{R}^N}p_t(\mathbf{x},\mathbf{y})|F(\mathbf{y})|\,dw(\mathbf{y}) \leq \frac{C}{w(B(\mathbf{x},t))}\|F\|_{L^1(dw)}.
\end{align*}
Moreover, by~\eqref{eq:balls_asymp}, for all $\mathbf{x} \in \mathbb{R}^N$ we have 
\begin{align*}
    \lim_{t \to \infty}\frac{1}{w(B(\mathbf{x},t))}=0.
\end{align*}
Consequently, by~\eqref{DtDxDyPoisson} we get that for fixed $\mathbf{x} \in \mathbb{R}^N$ we have
$\varphi(\mathbf{x},\cdot) \in C^2((0,\infty))$ and 
\begin{align*}
    \lim_{t \to \infty} \varphi(\mathbf{x},t)=\lim_{t \to \infty}t\partial_{t}\varphi(\mathbf{x},t)=0.
\end{align*}
Therefore, by the fundamental theorem of calculus and Theorem~\ref{teo:Markov}, for all $\mathbf{x} \in \mathbb{R}^N$ we have
\begin{equation}\label{eq:varphi_tt}
    R_jf(\mathbf{x})g(\mathbf{x})=\varphi(\mathbf{x},0)=\int_0^{\infty}t\partial_t^2\varphi(\mathbf{x},t)\,dt.
\end{equation}
Since, by the definition of $\{P_t\}_{t \geq 0}$, $\partial_{t}P_t=\sqrt{-\Delta_k}P_t$, and the operator $\sqrt{-\Delta_k}$ is self--adjoint on $L^2(dw)$, by~\eqref{eq:varphi_tt} we have
\begin{equation}\label{eq:varphi_long}
\begin{split}
    &\int_{\mathbb{R}^N}R_jf(\mathbf{x})g(\mathbf{x})\,dw(\mathbf{x})=\int_{\mathbb{R}^N}\int_0^{\infty}t\partial_t^2\varphi(\mathbf{x},t)\,dt\,dw(\mathbf{x})\\&=\int_{\mathbb{R}^N}\int_0^{\infty}t\big((\partial_t^2P_t)R_jf(\mathbf{x})P_tg(\mathbf{x})+2\partial_tP_tR_jf(\mathbf{x})\partial_tP_tg(\mathbf{x})+P_tR_jf(\mathbf{x})(\partial_t^2P_t)g(\mathbf{x})\big)\,dt\,dw(\mathbf{x})\\&=4\int\limits_{\mathbb{R}^N}\int\limits_0^{\infty}t\sqrt{-\Delta_k}P_tR_jf(\mathbf{x})\partial_tP_tg(\mathbf{x})\,dt\,dw(\mathbf{x}).
\end{split}
\end{equation}
Finally, note that by the definition of the Riesz transform (see~\eqref{eq:Riesz_mult}),~\eqref{eq:half_Laplacian_on_Fourier_side},~\eqref{eq:Poisson_transform_form}, and Lemma~\ref{lem:transform_properties}~\eqref{numitem:transform_on_der}, for all $1 \leq j \leq N$ we have
\begin{align*}
    \sqrt{-\Delta_k}(P_tR_j)f(\mathbf{x})&=c_k^{-1}\int_{\mathbb{R}^N}(-\|\xi\|)e^{-t\|\xi\|}E(i\xi, \mathbf x)\mathcal{F}(R_jf)(\xi)\,dw(\xi)\\&=c_k^{-1}\int_{\mathbb{R}^N}(-\|\xi\|)e^{-t\|\xi\|}E(i\xi, \mathbf x)\frac{-i\xi_j}{\|\xi\|}\mathcal{F}f(\xi)\,dw(\xi)\\&=T_jP_tf(\mathbf{x}),
\end{align*}
so the claim follows by~\eqref{eq:varphi_long}.
\end{proof}

As the consequence of Proposition~\ref{propo:dual_1}, we obtain the following corollary.

\begin{corollary}\label{coro:part_1}
Let $p,q>1$ be such that $\frac{1}{p}+\frac{1}{q}=1$. Then for all $f \in \mathcal{S}(\mathbb{R}^N)$ we have
\begin{equation}
    \|\mathcal{R}f\|_{L^p(dw)}=4\sup_{g_j \in \mathcal{S}(\mathbb{R}^N),\; \left\|\|\mathbf{g}(\mathbf{y})\|\right\|_{L^q(dw(\mathbf{y}))} \leq 1}\left|\sum_{j=1}^{N}\int_{\mathbb{R}^N}\int_0^{\infty}t \partial_tP_tg_j(\mathbf{x})T_{j}P_tf(\mathbf{x})\,dt\,dw(\mathbf{x})\right|.
\end{equation}
Here and subsequently, for $g_j \in \mathcal{S}(\mathbb{R}^N)$, $1 \leq j \leq N$, and $\mathbf{x} \in \mathbb{R}^N$ we denote
\begin{align*}
    \mathbf{g}(\mathbf{x})=(g_1(\mathbf{x}),\ldots,g_N(\mathbf{x})),
\end{align*}
\begin{align*}
    \|\mathbf{g}(\mathbf{x})\|=\Big(\sum_{j=1}^N|g_j(\mathbf{x})|^2\Big)^{1/2}.
\end{align*}
\end{corollary}
\section{Bellman function}\label{sec:Bellman}
In this section, we introduce the Bellman function, which will be the main ingredient of the proof of Theorem~\ref{teo:main}.
\begin{definition}\normalfont
Let $p \geq 2$ and let $q$ be such that $\frac{1}{p}+\frac{1}{q}=1$. Let $N_1,N_2 \in \mathbb{N}$. We define the \textit{Bellman function} $\beta:[0,\infty)^2 \to [0,\infty)$ by the formula
\begin{equation}\label{eq:beta}
    \beta(s,t) = s^p + t^q + \gamma
	\begin{cases}
		s^2 t^{2-q} \quad &\text{if } s^p < t^q \\
		\frac{2}{p} s^p + \left( \frac{2}{q} - 1 \right) t^q \quad &\text{if } s^p \geq t^q
	\end{cases}, 
	\quad \gamma := \frac{q(q-1)}{8}.
\end{equation}
The number $\gamma$ will be fixed throughout the paper. 
Next, we define the \textit{Nazarov-Treil Bellman function} $B:\mathbb{R}^{N_1} \times \mathbb{R}^{N_2} \to [0,\infty)$ by the formula
\begin{equation}\label{eq:B}
    B(\eta,\zeta)=\frac{1}{2}\beta(\|\eta\|,\|\zeta\|).
\end{equation}
\end{definition}

The function $B(\eta,\zeta)$ was introduced by Nazarov and Treil in~\cite{NazarovTreil}, then used and simplified in~\cite{CD1,CD2,DV1,DV2,Sch}.

Note that the function $B$ is differentiable but not smooth. We will need the smooth version of $B$. For $N_1,N_2 \in \mathbb{N}$ let $\phi:\mathbb{R}^{N_1} \times \mathbb{R}^{N_2} \to [0,\infty)$ be a smooth radial function supported in $B(0,1) \subset \mathbb{R}^{N_1} \times \mathbb{R}^{N_2}$ defined by the formula
 \begin{align*}
     \phi(\mathbf{x}_1,\mathbf{x}_2)=c_{N_1,N_2}\chi_{B(0,1)}(\mathbf{x}_1,\mathbf{x}_2)\exp(-(1-\|\mathbf{x}_1\|^2-\|\mathbf{x}_2\|^2)^{-1}),
 \end{align*}
 where $c_{N_1,N_2}>0$ is a constant such that
 \begin{align*}
     \int_{\mathbb{R}^{N_1} \times \mathbb{R}^{N_2}}\phi(\mathbf{x}_1,\mathbf{x}_2)\,d\mathbf{x}_1\,d\mathbf{x}_2=1.
 \end{align*}
For $\kappa>0$  and $(\mathbf{x}_1,\mathbf{x}_2) \in \mathbb{R}^{N_1} \times \mathbb{R}^{N_2}$ we set
 \begin{equation}\label{eq:phi}
     \phi_{\kappa}(\mathbf{x}_1,\mathbf{x}_2)=\frac{1}{\kappa^{N_1+N_2}}\phi(\mathbf{x}_1/{\kappa},\mathbf{x}_2/{\kappa}).
 \end{equation}
 
\begin{definition} Let $p \geq 2$ and let $q$ be such that $\frac{1}{p}+\frac{1}{q}=1$. Let $N_1,N_2 \in \mathbb{N}$ and $\kappa>0$. We define $B_{\kappa}:\mathbb{R}^{N_1} \times \mathbb{R}^{N_2} \to [0,\infty)$ by the formula
 \begin{equation}\label{eq:B_k}
     B_{\kappa}(\eta,\zeta)=B \star \phi_{\kappa}(\eta,\zeta):=\frac{1}{2}\beta_{\kappa}(\|\eta\|,\|\zeta\|)=\int_{\mathbb{R}^{N_1} \times \mathbb{R}^{N_2}}\phi_{\kappa}(\eta-\eta_1,\zeta-\zeta_1)B(\eta_1,\zeta_1)\,d\eta_1\,d\zeta_1.
 \end{equation}
\end{definition}

\begin{remark}\normalfont
In order to avoid misunderstanding, we would like to emphasise that the convolution "$\star$" in~\eqref{eq:B_k} is the ordinary one (not the Dunkl generalized convolution). Let us also point out that in the proof of Theorem~\ref{teo:main} we will set $N_1=1$ and $N_2=N$.
\end{remark}

The following properties of the functions $\beta_{\kappa}$ and $B_{\kappa}$ were proved in~\cite[Theorems 3 and 4]{Sch} and~\cite{Mauceri_arxiv}.

\begin{proposition}\label{propo:gradient_bellman}
Let $p \geq 2$ and let $q$ be such that $\frac{1}{p}+\frac{1}{q}=1$. There is a constant $C_p>0$ such that for all $\kappa \in (0,1]$ and $s,t>0$ we have
\begin{align*}
    0 \leq \partial_{s}\beta_{\kappa}(s,t) \leq C_p\max((s+\kappa)^p,(t+\kappa)),
\end{align*}
\begin{align*}
    0 \leq \partial_{t}\beta_{\kappa}(s,t) \leq C_p(t+\kappa)^{q-1}.
\end{align*}
\end{proposition}
\begin{theorem}\label{teo:bellman}
Let $p \geq 2$ and let $q$ be such that $\frac{1}{p}+\frac{1}{q}=1$. Let $\kappa \in (0,1]$. Then $B_{\kappa} \in C^{\infty}(\mathbb{R}^{N_1} \times \mathbb{R}^{N_2})$. Moreover, there is a function $\tau:\mathbb{R}^{N_1} \times \mathbb{R}^{N_2} \to [0,\infty)$ such that for all $\eta \in \mathbb{R}^{N_1}$, $\zeta \in \mathbb{R}^{N_2}$, and $\omega=(\omega_1,\omega_2) \in \mathbb{R}^{N_1} \times \mathbb{R}^{N_2}$ we have
\begin{equation}\label{eq:bell_1}
    0 \leq B_{\kappa}(\eta,\zeta) \leq \frac{1+\gamma}{2}\big((\|\eta\|+\kappa)^{p}+(\|\zeta\|+\kappa)^q\big),
\end{equation}
\begin{equation}\label{eq:bell_2}
    \langle \Hess(B_{\kappa})(\eta,\zeta)\omega,\omega\rangle \geq \frac{\gamma}{2} \big((\tau \star \phi_{\kappa})(\eta,\zeta)\|\omega_1\|^2+(\frac{1}{\tau} \star \phi_{\kappa})(\eta,\zeta)\|\omega_2\|^2\big).
\end{equation}
It follows from the proof of~\cite[Theorem 3]{Sch} that one can take
\begin{equation}\label{eq:tau}
    \tau(\eta,\zeta)=\|\zeta\|^{2-q}.
\end{equation}
\end{theorem}

\begin{remark}\normalfont
In our further considerations, we will need the explicit form of $\tau$ (see~\eqref{eq:tau}). This form of $\tau$ follows directly from the proofs presented in~\cite[Theorem 3]{Sch} and~\cite[Proposition 6.3]{Mauceri_arxiv}, although it is not given explicitly there. Therefore, for the convenience  of the reader, we repeat the proof from~\cite{Mauceri_arxiv} in Appendix~\ref{appendix} with $\tau$ given by~\eqref{eq:tau}.
\end{remark}

In our further consideration, we will need the following elementary lemma, which concerns the properties of $\tau$ in~\eqref{eq:tau}.

\begin{lemma}
Let $1 <q \leq 2$ and $N_3 \in \mathbb{N}$. Then for all $\mathbf{a},\mathbf{b} \in \mathbb{R}^{N_3}$ we have
\begin{equation}\label{eq:elem_1}
    \int_0^{1}s\|s\mathbf{a}+(1-s)\mathbf{b}\|^{2-q}\,ds \geq 2^{-6}\max(\|\mathbf{a}\|,\|\mathbf{b}\|)^{2-q},
\end{equation}
\begin{equation}\label{eq:elem_2}
    \int_{0}^{1}s\|s\mathbf{a}+(1-s)\mathbf{b}\|^{q-2} \,ds \geq 2^{-1}\max(\|\mathbf{a}\|,\|\mathbf{b}\|)^{q-2}.
\end{equation}
\end{lemma}

\begin{proof}
The proof is standard, but we provide it for the sake of completeness. We will prove~\eqref{eq:elem_1} first. Let us consider two cases.
\\
\textbf{Case 1.} $\|\mathbf{a}\| \geq \|\mathbf{b}\|$. Then we have
\begin{align*}
    &\int_0^{1}s\|s\mathbf{a}+(1-s)\mathbf{b}\|^{2-q}\,ds \geq \int_{3/4}^{1}s\|s\mathbf{a}+(1-s)\mathbf{b}\|^{2-q}\,ds \geq  \int_{3/4}^{1}s(s\|\mathbf{a}\|-(1-s)\|\mathbf{b}\|)^{2-q}\,ds \\&\geq \int_{3/4}^{1}s(3\|\mathbf{a}\|/4-\|\mathbf{b}\|/4)^{2-q}\,ds \geq 2^{-6}\|\mathbf{a}\|^{2-q}.
\end{align*}
\\
\textbf{Case 2.} $\|\mathbf{b}\|>\|\mathbf{a}\|$. By the change of variables we have
\begin{align*}
    \int_0^{1}s\|s\mathbf{a}+(1-s)\mathbf{b}\|^{2-q}\,ds = \int_0^{1}s\|s\mathbf{b}+(1-s)\mathbf{a}\|^{2-q}\,ds,
\end{align*}
so we are reduced to Case 1.
\\
In order to prove~\eqref{eq:elem_2}, we write
\begin{align*}
    &\int_0^{1}s\|s\mathbf{a}+(1-s)\mathbf{b}\|^{q-2}\,ds \geq \int_0^{1}s(s\|\mathbf{a}\|+(1-s)\|\mathbf{b}\|)^{q-2}\,ds \geq \max(\|\mathbf{a}\|,\|\mathbf{b}\|)^{q-2}\int_{0}^1s\,ds \\&=2^{-1}\max(\|\mathbf{a}\|,\|\mathbf{b}\|)^{q-2}.
\end{align*}
\end{proof}

\subsection{Dunkl Laplacian on Bellman function}

\begin{definition}\normalfont\label{def:u}
Let $p \geq 2$ and let $q$ be such that $\frac{1}{p}+\frac{1}{q}=1$, $\kappa \in (0,1]$. For $f \in L^p(dw)$ and $g_j \in L^q(dw)$, $1 \leq j \leq N$, $\mathbf{x} \in \mathbb{R}^N$, and $t>0$ we define
\begin{equation}\label{eq:u}
    u(\mathbf{x},t):=(P_tf(\mathbf{x}),P_tg_1(\mathbf{x}),\ldots,P_tg_{N}(\mathbf{x})),
\end{equation}
\begin{equation}\label{eq:u_tilde}
    \widetilde{u}(\mathbf{x},t):=(P_tf(\mathbf{x}),P_t\mathbf{g}(\mathbf{x}))=\big(P_tf(\mathbf{x}),(P_tg_1(\mathbf{x}),\ldots,P_tg_{N}(\mathbf{x}))\big),
\end{equation}
\begin{equation}\label{eq:b}
    b_{\kappa}(\mathbf{x},t):=B_{\kappa}(P_tf(\mathbf{x}),P_t\mathbf{g}(\mathbf{x}))=B_{\kappa}\big(P_tf(\mathbf{x}),(P_tg_1(\mathbf{x}),\ldots,P_tg_{N}(\mathbf{x})))\big),
\end{equation}
where $\{P_t\}_{t \geq 0}$ is the Dunkl Poisson semigroup (see Definition~\ref{def:Poisson_semigroup}).
\end{definition}

\begin{lemma}\label{lem:regularity}
Assume that $f,g_j \in \mathcal{S}(\mathbb{R}^N)$, $1 \leq j \leq N$, and $\kappa \in (0,1]$. Then
\begin{enumerate}[(A)]
    \item{$b_{\kappa} \in C^{\infty}(\mathbb{R}^N \times (0,\infty))$;}\label{eq:numitem:smooth}
    \item{there is a constant $C_{f,\mathbf{g}}>0$, which depends on $f$ and $\mathbf{g}$ and is independent of $\kappa$, such that for all $\mathbf{x} \in \mathbb{R}^N$ and $t>0$ we have
    \begin{equation}
        |\partial_{t}b_{\kappa}(\mathbf{x},t)| \leq \frac{1}{t}C_{f,\mathbf{g}}.
    \end{equation}
    }\label{numitem:t_growth}
\end{enumerate}
\end{lemma}

\begin{proof}
By Lemma~\ref{lem:basic}, for $f,g_j \in \mathcal{S}(\mathbb{R}^N)$, $1 \leq j \leq N$, the functions $P_tf,P_tg_j$ belong to $C^{\infty}(\mathbb{R}^N \times (0,\infty))$. Therefore, by Theorem~\ref{teo:bellman} and~\eqref{eq:b}, $b_{\kappa}$ is a composition of smooth functions, so~\eqref{eq:numitem:smooth} follows. In order to prove~\eqref{numitem:t_growth}, note that by the chain rule we have
\begin{align*}
    \partial_tb_{\kappa}(\mathbf{x},t)=\langle \nabla B_{\kappa}(\widetilde{u}(\mathbf{x},t)),\partial_t u(\mathbf{x},t)\rangle.
\end{align*}
Consequently, by Proposition~\ref{propo:gradient_bellman} and the Cauchy--Schwarz inequality we get that there is a constant $C_p>0$, which depends just on $p$, such that
\begin{equation}\label{eq:first}
    |\partial_tb_{\kappa}(\mathbf{x},t)| \leq C_p\|\partial_{t}u(\mathbf{x},t)\|\left(|P_tf(\mathbf{x})|^{p-1}+\|P_t\mathbf{g}(\mathbf{x})\|^{q-1}+\|P_t\mathbf{g}(\mathbf{x})\|+\kappa^{q-1}\right).
\end{equation}
Note that by Lemma~\ref{lem:basic} and the fact that $f,g_j \in \mathcal{S}(\mathbb{R}^N)$ there is a constant $C=C_{f,\mathbf{g}}>0$ such for all $\mathbf{x} \in \mathbb{R}^N$, $t>0$, and $1 \leq j \leq N$ we have
\begin{align*}
    |P_tf(\mathbf{x})| \leq C, \ \ \ |P_tg_{j}(\mathbf{x})| \leq C,
\end{align*}
\begin{align*}
    \|\partial_{t}u(\mathbf{x},t)\| \leq \frac{C}{t},
\end{align*}
so, by~\eqref{eq:first}, the proof of~\eqref{numitem:t_growth} is finished.
\end{proof}

In the next proposition we obtain an explicit formula for $\Delta_k b_{\kappa}$ (cf.~\cite[Section 4]{Graczyk}).

\begin{proposition}\label{propo:laplace_on_b}
Assume that $f,g_j \in \mathcal{S}(\mathbb{R}^N)$, $1 \leq j \leq N$, and $\kappa \in (0,1]$. Let $u$, $\widetilde{u}$, and $b_{\kappa}$ be as in Definition~\ref{def:u}. Then for all $\mathbf{x} \in \mathbb{R}^N$ and $t>0$ we have
\begin{equation}\label{eq:delta_on_bellman}
\begin{split}
    &(\partial_t^2+\Delta_{k,\mathbf{x}})b_{\kappa}(\mathbf{x},t)=\langle \Hess(B_{\kappa})(\widetilde{u}(\mathbf{x},t))\partial_{t}u(\mathbf{x},t),\partial_t u(\mathbf{x},t)\rangle\\&+\sum_{j=1}^{N}\langle \Hess(B_{\kappa})(\widetilde{u}(\mathbf{x},t))\partial_{j,\mathbf{x}}u(\mathbf{x},t),\partial_{j, \mathbf{x}} u(\mathbf{x},t)\rangle\\&+\sum_{\alpha \in R}k(\alpha)\int_0^{1}s\left\langle \Hess(B_{\kappa})(s\widetilde{u}(\mathbf{x},t)+(1-s)\widetilde{u}(\sigma_{\alpha}(\mathbf{x}),t))\rho_{\alpha}u(\mathbf{x},t),\rho_{\alpha}u(\mathbf{x},t)\right\rangle\,ds,
\end{split}
\end{equation}
where
\begin{equation}\label{eq:small_delta}
    \rho_{\alpha}u(\mathbf{x},t):=\frac{u(\mathbf{x},t)-u(\sigma_{\alpha}(\mathbf{x}),t)}{\langle \mathbf{x},\alpha\rangle}.
\end{equation}
\end{proposition}

\begin{proof}
It follows by the chain rule (see e.g.~\cite[Lemma 1.4]{DV1}) that
\begin{equation}\label{eq:laplace_calc_1_1}
\begin{split}
     \partial_t^2b_{\kappa}(\mathbf{x},t)&=\langle \Hess(B_{\kappa})(\widetilde{u}(\mathbf{x},t))\partial_{t}u(\mathbf{x},t),\partial_t u(\mathbf{x},t)\rangle\\&+\langle \nabla B_{k}(\widetilde{u}(\mathbf{x},t)),\partial_t^2u(\mathbf{x},t)\rangle,
\end{split}
\end{equation}
and
\begin{equation}\label{eq:laplace_calc_1_2}
\begin{split}
     \Delta_{\mathbf{x}}b_{\kappa}(\mathbf{x},t)&=\sum_{j=1}^{N}\langle \Hess(B_{\kappa})(\widetilde{u}(\mathbf{x},t))\partial_{j,\mathbf{x}}u(\mathbf{x},t),\partial_{j, \mathbf{x}} u(\mathbf{x},t)\rangle\\&+\langle \nabla B_{k}(\widetilde{u}(\mathbf{x},t)),\Delta_{\mathbf{x}} u(\mathbf{x},t)\rangle.
\end{split}
\end{equation}
Moreover, for $\alpha \in R$ we have
\begin{equation}\label{eq:laplace_calc_2}
    \begin{split}
        \frac{\partial_{\alpha,\mathbf{x}}b_{\kappa}(\mathbf{x},t)}{\langle \mathbf{x},\alpha\rangle}=\left\langle \nabla B_{k}(\widetilde{u}(\mathbf{x},t)),\frac{ \partial_{\alpha,\mathbf{x}}u(\mathbf{x},t)}{\langle \mathbf{x},\alpha\rangle}\right\rangle.
    \end{split}
\end{equation}
Recall that, by the definition of the Dunkl Poisson semigroup (see Definition~\ref{def:Poisson_semigroup}), $(\partial_t^2+\Delta_{k,\mathbf{x}})u(\mathbf{x},t)=0$. Therefore, by~\eqref{eq:laplace_formula},~\eqref{eq:laplace_calc_1_1},~\eqref{eq:laplace_calc_1_2}, and~\eqref{eq:laplace_calc_2} we get
\begin{align*}
    &(\partial_t^2+\Delta_{k,\mathbf{x}})b_{\kappa}(\mathbf{x},t)=\langle \Hess(B_{\kappa})(\widetilde{u}(\mathbf{x},t))\partial_{t}u(\mathbf{x},t),\partial_t u(\mathbf{x},t)\rangle\\&+\sum_{j=1}^{N}\langle \Hess(B_{\kappa})(\widetilde{u}(\mathbf{x},t))\partial_{j,\mathbf{x}}u(\mathbf{x},t),\partial_{j, \mathbf{x}} u(\mathbf{x},t)\rangle\\&+\sum_{\alpha \in R}k(\alpha)\left\langle \nabla B_{k}(\widetilde{u}(\mathbf{x},t)),\frac{u(\mathbf{x},t)-u(\sigma_{\alpha}(\mathbf{x}),t)}{\langle \mathbf{x},\alpha\rangle^2}\right\rangle\\&-\sum_{\alpha \in R}k(\alpha)\frac{b_{\kappa}(\mathbf{x},t)-b_{\kappa}(\sigma_{\alpha}(\mathbf{x}),t)}{\langle \mathbf{x},\alpha\rangle^2}.
\end{align*}
Finally, note that by the Taylor's expansion of the function $b_{\kappa}(\mathbf{x},t)$, for all $\alpha \in R$ we have
\begin{align*}
    &\left\langle \nabla B_{k}(\widetilde{u}(\mathbf{x},t)),\frac{u(\mathbf{x},t)-u(\sigma_{\alpha}(\mathbf{x}),t)}{\langle \mathbf{x},\alpha\rangle^2}\right\rangle-\frac{b_{\kappa}(\mathbf{x},t)-b_{\kappa}(\sigma_{\alpha}(\mathbf{x}),t)}{\langle \mathbf{x},\alpha\rangle^2}\\&=\int_0^{1}s\left\langle\Hess(B_{\kappa})(s\widetilde{u}(\mathbf{x},t)+(1-s)\widetilde{u}(\sigma_{\alpha}(\mathbf{x}),t)) \rho_{\alpha}u(\mathbf{x},t), \rho_{\alpha}u(\mathbf{x},t)\right\rangle\,ds,
\end{align*}
so the proof is finished.
\end{proof}

\begin{corollary}\label{coro:can_change_order}
Assume that $f,g_j \in \mathcal{S}(\mathbb{R}^N)$, $1 \leq j \leq N$. There is a constant $C=C_{f,\mathbf{g}}>0$ such that for all $\mathbf{x} \in \mathbb{R}^N$, $\kappa \in (0,1]$, and $t>0$ we have
\begin{align*}
    |\Delta_{k,\mathbf{x}}(b_{\kappa})(\mathbf{x},t)| \leq \frac{C_{f,\mathbf{g}}}{t^2}.
\end{align*}
\end{corollary}

\begin{proof}
Since $f,g_j \in \mathcal{S}(\mathbb{R}^N)$, $1 \leq j \leq N$ , by Lemma~\ref{lem:basic} there is a constant $C>0$, which depends on $f,g_j$ such that for all $\mathbf{x} \in \mathbb{R}^N$ and $t>0$ we have
\begin{align*}
    |P_tf(\mathbf{x})| \leq C, \ \ \ \|P_t\mathbf{g}(\mathbf{x})\| \leq C.
\end{align*}
Consequently, by the fact that $\nabla B_{\kappa}$ and $\Hess(B_{\kappa})$ are smooth and~\eqref{eq:b}, we obtain that there is a constant $C'=C'_{f,\mathbf{g}}>0$ such that for all $\mathbf{x}\in \mathbb{R}^N$, $t>0$, $\alpha \in R$, and $s \in [0,1]$ we have
\begin{align*}
    \|\nabla B_{\kappa}(\widetilde{u}(\mathbf{x},t))\| \leq C', \ \ \ \|\Hess(B_{\kappa})(s\widetilde{u}(\mathbf{x},t)+(1-s)\widetilde{u}(\sigma_{\alpha}(\mathbf{x}),t))\|_{{\rm HS}} \leq C',
\end{align*}
where $\|\cdot\|_{{\rm HS}}$ is the Hilbert--Schmith norm. Moreover, by Lemma~\ref{lem:basic}, there is a constant $C''=C''_{f,\mathbf{g}}>0$ such that for all $\mathbf{x} \in \mathbb{R}^N$ and $t>0$ we have
\begin{align*}
    |\partial_tu(\mathbf{x},t)| \leq \frac{C''}{t}, \ \ \ |\partial_t^2u(\mathbf{x},t)| \leq \frac{C''}{t^2}.
\end{align*}
Recall that for all $\mathbf{x} \in \mathbb{R}^N$ and $\alpha \in R$ we have $\sqrt{2}|\langle \mathbf{x},\alpha\rangle|=\|\mathbf{x}-\sigma_{\alpha}(\mathbf{x})\|$ (see~\eqref{eq:x_minus_sigma_x}). Hence, by~\eqref{DtDxDyPoisson} and the mean value theorem we have that there is a constant $C'''=C'''_{f,\mathbf{g}}>0$ such that for all $\mathbf{x} \in \mathbb{R}^N$ and $t>0$ we have
\begin{align*}
    |\rho_{\alpha}u(\mathbf{x},t)| \leq \frac{C'''}{t}.
\end{align*}
Finally, the claim is a consequence of~\eqref{eq:delta_on_bellman},~\eqref{eq:laplace_calc_1_1}, and the Cauchy--Schwarz inequality.
\end{proof}

\section{Proof of Theorem{~\ref{teo:main}}}

In this section, we prove Theorem~\ref{teo:main}. We closely follow the reasoning from~\cite{CD1} and~\cite{Sch}.

\subsection{Definition of \texorpdfstring{$I(n,\varepsilon,\kappa)$}{I(n,e,k)}}\label{sec:I}

The proof of Theorem~\ref{teo:main} is based on the upper and lower estimates of the quantities $I(n,\varepsilon,\kappa)$, which approximate the integral
\begin{align*}
\int_{\mathbb{R}^N}\int_0^{\infty}t(\partial_t^2+\Delta_{k,\mathbf{x}})b_{\kappa}(\mathbf{x},t)\,dt\,dw(\mathbf{x}).    
\end{align*}

\begin{definition}\label{def:Phi}\normalfont
Let $\Phi \in C^{\infty}_c(\mathbb{R}^N)$ be a radial radially decreasing function such that $\supp \Phi \subseteq B(0,2)$, $0 \leq \Phi \leq 1$, and $\Phi(\mathbf{x})=1$ for all $\mathbf{x} \in B(0,1)$. The function $\Phi$ will be fixed throughout the paper.

For $a>0$ we define the function $\nu_{a}:(0,\infty) \to (0,\infty)$ by the formula
\begin{equation}\label{eq:nu}
    \nu_a(t):=t\exp(-a(t+t^{-1})).
\end{equation}\label{eq:II}
Let $n \in \mathbb{N}$ and $\varepsilon>0$. For a function $\kappa:\mathbb{N} \to (0,1]$ we set
\begin{equation}\label{eq:I}
    I(n,\varepsilon,\kappa):=\int_{\mathbb{R}^N}\Phi(\mathbf{x}/n)\int_0^{\infty}\nu_{\varepsilon}(t)(\partial_t^2+\Delta_{k,\mathbf{x}})(b_{\kappa(n)})(\mathbf{x},t)\,dt\,dw(\mathbf{x}).
\end{equation}
\end{definition}

\subsection{Lower estimate of \texorpdfstring{$I(n,\varepsilon,\kappa)$}{I(n,e,k)}}

\begin{lemma}\label{lem:part_2}
Assume that $f ,g_j \in \mathcal{S}(\mathbb{R}^N)$, $1 \leq j \leq N$. Then for any $\kappa:\mathbb{N} \to (0,1]$ we have
\begin{equation}\label{eq:noninv_version}
    \sum_{j=1}^{N}\int_{\mathbb{R}^N}\int_0^{\infty}t|\partial_tP_tg_j(\mathbf{x})T_jP_tf(\mathbf{x})|\,dt\,dw(\mathbf{x}) \leq \frac{2}{\gamma}\left(\sum_{\alpha \in R}k(\alpha)+2^7\right)\liminf_{\varepsilon \to 0^{+}}\liminf_{n \to \infty}I(n,\varepsilon,\kappa).
\end{equation}
Moreover, if the function $f$ is $G$-invariant, then
\begin{equation}\label{eq:inv_version}
    \sum_{j=1}^{N}\int_{\mathbb{R}^N}\int_0^{\infty}t|\partial_tP_tg_j(\mathbf{x})T_jP_tf(\mathbf{x})|\,dt\,dw(\mathbf{x}) \leq \frac{2}{\gamma}\liminf_{\varepsilon \to 0^{+}}\liminf_{n \to \infty}I(n,\varepsilon,\kappa).
\end{equation}
\end{lemma}

\begin{proof}
Let us prove~\eqref{eq:noninv_version} first. Fix $\kappa:\mathbb{N} \to (0,1]$. By the monotone convergence theorem we have
\begin{align*}
    &\sum_{j=1}^{N}\int_{\mathbb{R}^N}\int_0^{\infty}t|\partial_tP_tg_j(\mathbf{x})T_jP_tf(\mathbf{x})|\,dt\,dw(\mathbf{x})\\& = \lim_{\varepsilon \to 0^{+}}\lim_{n \to \infty} \sum_{j=1}^{N}\int_{\mathbb{R}^N}\Phi(\mathbf{x}/n)\int_0^{\infty}\nu_{\varepsilon}(t)|\partial_tP_tg_j(\mathbf{x})T_jP_tf(\mathbf{x})|\,dt\,dw(\mathbf{x}).
\end{align*}
Fix $n \in \mathbb{N}$ and $\varepsilon>0$. By the definition of $T_j$ (see~\eqref{eq:Dunkl_op}) we get
\begin{equation}\label{eq:T_j_split}
\begin{split}
    &\sum_{j=1}^{N}\int_{\mathbb{R}^N}\Phi(\mathbf{x}/n)\int_0^{\infty}\nu_{\varepsilon}(t)|\partial_tP_tg_j(\mathbf{x})T_jP_tf(\mathbf{x})|\,dt\,dw(\mathbf{x}) \\&\leq \sum_{j=1}^{N}\int_{\mathbb{R}^N}\Phi(\mathbf{x}/n)\int_0^{\infty}\nu_{\varepsilon}(t)|\partial_tP_tg_j(\mathbf{x})\partial_{j,\mathbf{x}}P_tf(\mathbf{x})|\,dt\,dw(\mathbf{x}) \\&+\sum_{j=1}^{N}\sum_{\alpha \in R}\frac{k(\alpha)}{2}|\alpha_j|\int_{\mathbb{R}^N}\Phi(\mathbf{x}/n)\int_0^{\infty}\nu_{\varepsilon}(t)\left|\partial_tP_tg_j(\mathbf{x})\frac{P_tf(\mathbf{x})-P_tf(\sigma_{\alpha}(\mathbf{x}))}{\langle \mathbf{x},\alpha\rangle}\right|\,dt\,dw(\mathbf{x})\\&=:I_1+I_2.
\end{split}
\end{equation}
We will estimate $I_1$ and $I_2$ separately.
\\
\noindent
\textbf{Estimate of $I_1$.} In order to estimate $I_1$, for $\mathbf{y}=(y_1,\mathbf{y}_2)\in \mathbb{R} \times \mathbb{R}^N$ we set
\begin{equation}\label{eq:tau_1}
    \tau_1(\mathbf{y})=\tau_1(y_1,\mathbf{y}_2)=\|\mathbf{y}_2\|^{2-q}.
\end{equation}
Recall that $\phi_{\kappa(n)}$ for $n \in \mathbb{N}$ is defined in~\eqref{eq:phi}. By the fact that $\int_{\mathbb{R}^{N+1}}\phi_{\kappa(n)}(\mathbf{y})\,d\mathbf{y}=1$, the inequality between the arithmetic and geometric mean, and~\eqref{eq:bell_2} with
\begin{align*}
\begin{cases}
    N_1=1, \, N_2=N,\\
    \eta=P_tf(\mathbf{x}),\\
    \zeta=P_t\mathbf{g}(\mathbf{x}),\\
    \omega=\partial_{t}u(\mathbf{x},t)\text{ or }\omega=\partial_{j,\mathbf{x}}u(\mathbf{x},t)
\end{cases}
\end{align*}
(see~\eqref{eq:u}) we get
\begin{equation}\label{eq:long_1}
\begin{split}
    &\sum_{j=1}^{N}\int_{\mathbb{R}^N}\Phi(\mathbf{x}/n)\int_0^{\infty}\nu_{\varepsilon}(t)|\partial_tP_tg_j(\mathbf{x})\partial_{j,\mathbf{x}}P_tf(\mathbf{x})|\,dt\,dw(\mathbf{x}) \\&= \sum_{j=1}^{N}\int_{\mathbb{R}^N}\Phi(\mathbf{x}/n)\int_0^{\infty}\nu_{\varepsilon}(t)\left(\int_{\mathbb{R}^{N+1}}\phi_{\kappa(n)}(u(\mathbf{x},t)-\mathbf{y})\,d\mathbf{y}\right)|\partial_tP_tg_j(\mathbf{x})\partial_{j,\mathbf{x}}P_tf(\mathbf{x})|\,dt\,dw(\mathbf{x})\\&\leq \int_{\mathbb{R}^N}\Phi(\mathbf{x}/n)\int_0^{\infty}\nu_{\varepsilon}(t)\left(\int_{\mathbb{R}^{N+1}}\phi_{\kappa(n)}(u(\mathbf{x},t)-\mathbf{y})\left(\tau_1(\mathbf{y})^{-1}\left(\sum_{j=1}^{N}|\partial_tP_tg_j(\mathbf{x})|^2\right)\right)\,d\mathbf{y}\right)\,dt\,dw(\mathbf{x})\\&+ \int_{\mathbb{R}^N}\Phi(\mathbf{x}/n)\int_0^{\infty}\nu_{\varepsilon}(t)\left(\int_{\mathbb{R}^{N+1}}\phi_{\kappa(n)}(u(\mathbf{x},t)-\mathbf{y})\left(\tau_1(\mathbf{y})\left(\sum_{j=1}^{N}|\partial_{j,\mathbf{x}}P_tf(\mathbf{x})|^2\right)\right)\,d\mathbf{y}\right)\,dt\,dw(\mathbf{x}) \\&\leq \frac{2}{\gamma}\int_{\mathbb{R}^N}\Phi(\mathbf{x}/n)\int_0^{\infty}\nu_{\varepsilon}(t) \langle \Hess B_{\kappa(n)}(\widetilde{u}(\mathbf{x},t))\partial_tu_t(\mathbf{x}),\partial_tu_t(\mathbf{x}) \rangle\,dt\,dw(\mathbf{x})\\&+\sum_{j=1}^{N}\frac{2}{\gamma}\int_{\mathbb{R}^N}\Phi(\mathbf{x}/n)\int_0^{\infty}\nu_{\varepsilon}(t) \langle \Hess B_{\kappa(n)}(\widetilde{u}(\mathbf{x},t))\partial_{j,\mathbf{x}}u_t(\mathbf{x}),\partial_{j,\mathbf{x}}u_t(\mathbf{x}) \rangle\,dt\,dw(\mathbf{x}).
\end{split}
\end{equation}
\\
\noindent
\textbf{Estimate of $I_2$.} In order to estimate $I_2$, by the fact that $\|\alpha\|=\sqrt{2}$,   $\int_{\mathbb{R}^{N+1}}\phi_{\kappa(n)}(\mathbf{y})\,d\mathbf{y}=1$, and the inequality between the arithmetic and geometric mean, we have
\begin{equation}\label{eq:long_2}
    \begin{split}
        &\sum_{j=1}^{N}\sum_{\alpha \in R}\frac{k(\alpha)}{2}|\alpha_j|\int_{\mathbb{R}^N}\Phi(\mathbf{x}/n)\int_0^{\infty}\nu_{\varepsilon}(t)\left|\partial_tP_tg_j(\mathbf{x})\frac{P_tf(\mathbf{x})-P_tf(\sigma_{\alpha}(\mathbf{x}))}{\langle \mathbf{x},\alpha\rangle}\right|\,dt\,dw(\mathbf{x}) \\&\leq 
        \sum_{\alpha \in R}k(\alpha)\int_{\mathbb{R}^N}\Phi(\mathbf{x}/n)\int_0^{\infty}\nu_{\varepsilon}(t)\left(\sum_{j=1}^{N}|\partial_tP_tg_j(\mathbf{x})|^2\right)^{1/2}\left|\frac{P_tf(\mathbf{x})-P_tf(\sigma_{\alpha}(\mathbf{x}))}{\langle \mathbf{x},\alpha\rangle}\right|\,dt\,dw(\mathbf{x}) \\&\leq 
        \sum_{\alpha \in R}k(\alpha)\int_{\mathbb{R}^N}\Phi(\mathbf{x}/n)\int_0^{\infty}\nu_{\varepsilon}(t)\left(\int_{\mathbb{R}^{N+1}}\phi_{\kappa(n)}(u(\mathbf{x},t)-\mathbf{y})\tau_1(\mathbf{y})^{-1}\,d\mathbf{y}\right)\left(\sum_{j=1}^{N}|\partial_tP_tg_j(\mathbf{x})|^2\right)\,dt\,dw(\mathbf{x}) \\&+\sum_{\alpha \in R}k(\alpha)\int_{\mathbb{R}^N}\Phi(\mathbf{x}/n)\int_0^{\infty}\nu_{\varepsilon}(t)\left(\int_{\mathbb{R}^{N+1}}\phi_{\kappa(n)}(u(\mathbf{x},t)-\mathbf{y})\tau_1(\mathbf{y})\,d\mathbf{y}\right)\left|\frac{P_tf(\mathbf{x})-P_tf(\sigma_{\alpha}(\mathbf{x}))}{\langle \mathbf{x},\alpha\rangle}\right|^2\,dt\,dw(\mathbf{x})\\&=:I_{2,1}+I_{2,2}.
    \end{split}
\end{equation}

The summand $I_{2,1}$ is the same as the first summand in~\eqref{eq:long_1}, but it is multiplied by $\sum_{\alpha \in R}k(\alpha)$. Recall that $\phi_{\kappa(n)} \geq 0$. In order to estimate $I_{2,2}$ we write
\begin{align*}
    &\int_{\mathbb{R}^{N+1}}\phi_{\kappa(n)}(u(\mathbf{x},t)-\mathbf{y})\tau_1(\mathbf{y})\,d\mathbf{y} \leq \int_{\mathbb{R}^{N+1}}\left(\phi_{\kappa(n)}(u(\mathbf{x},t)-\mathbf{y})+\phi_{\kappa(n)}(u(\sigma_{\alpha}(\mathbf{x}),t)-\mathbf{y})\right)\tau_1(\mathbf{y})\,d\mathbf{y}\\&=\int_{\mathbb{R}^{N+1}}\phi_{\kappa(n)}(\mathbf{y})\left(\tau_1(u(\mathbf{x},t)-\mathbf{y})+\tau_1(u(\sigma_{\alpha}(\mathbf{x}),t)-\mathbf{y})\right)\,d\mathbf{y}\\&\leq 2\int_{\mathbb{R}^{N+1}}\phi_{\kappa(n)}(\mathbf{y})\max\left(\tau_1(u(\mathbf{x},t)-\mathbf{y}),\tau_1(u(\sigma_{\alpha}(\mathbf{x}),t)-\mathbf{y})\right)\,d\mathbf{y} ,
\end{align*}
then use~\eqref{eq:elem_1} with $\mathbf{a}=u(\mathbf{x},t)-\mathbf{y}$ and $\mathbf{b}=u(\sigma_{\alpha}(\mathbf{x}),t)-\mathbf{y}$. Consequently,

\begin{align*}
    &\int_{\mathbb{R}^{N+1}}\phi_{\kappa(n)}(u(\mathbf{x},t)-\mathbf{y})\tau_1(\mathbf{y})\,d\mathbf{y} \leq 2^7 \int_{\mathbb{R}^{N+1}}\phi_{\kappa(n)}(\mathbf{y})\int_0^1 s\tau_1\big(su(\mathbf{x},t)+(1-s)u(\sigma_{\alpha}(\mathbf{x}),t)-\mathbf{y}\big)\,ds \,d\mathbf{y},
\end{align*}
which leads us to
\begin{align*}
    &\sum_{\alpha \in R}k(\alpha)\int_{\mathbb{R}^N}\Phi(\mathbf{x}/n)\int_0^{\infty}\nu_{\varepsilon}(t)\left(\int_{\mathbb{R}^{N+1}}\phi_{\kappa(n)}(u(\mathbf{x},t)-\mathbf{y})\tau_1(\mathbf{y})\,d\mathbf{y}\right)\left|\frac{P_tf(\mathbf{x})-P_tf(\sigma_{\alpha}(\mathbf{x}))}{\langle \mathbf{x},\alpha\rangle}\right|^2\,dt\,dw(\mathbf{x}) \\&\leq 2^7\sum_{\alpha \in R}k(\alpha)\int_{\mathbb{R}^N}\Phi(\mathbf{x}/n)\int_0^{\infty}\nu_{\varepsilon}(t)\left(\int_0^1 s(\tau_1 \star \phi_{\kappa(n)})(su(\mathbf{x},t)+(1-s)u(\sigma_{\alpha}(\mathbf{x}),t)) \,ds\right)\\&\times\left|\frac{P_tf(\mathbf{x})-P_tf(\sigma_{\alpha}(\mathbf{x}))}{\langle \mathbf{x},\alpha\rangle}\right|^2\,dt\,dw(\mathbf{x}).
\end{align*}
Now, by~\eqref{eq:bell_2} with
\begin{align*}
\begin{cases}
    N_1=1, \, N_2=N,\\
    \eta=sP_tf(\mathbf{x})+(1-s)P_tf(\sigma_{\alpha}(\mathbf{x})),\\
    \zeta=sP_t\mathbf{g}(\mathbf{x})+(1-s)P_t\mathbf{g}(\sigma_{\alpha}(\mathbf{x})), \\
    \omega=\rho_{\alpha}u(\mathbf{x},t),
\end{cases}
\end{align*}
where $s \in [0,1]$ and $\alpha \in R$ (see~\eqref{eq:u},~\eqref{eq:u_tilde}, and~\eqref{eq:small_delta}) we get
\begin{equation}\label{eq:long_4}
\begin{split}
   & 2^7\sum_{\alpha \in R}k(\alpha)\int_{\mathbb{R}^N}\Phi(\mathbf{x}/n)\int_0^{\infty}\nu_{\varepsilon}(t)\left(\int_0^1 s(\tau_1 \star \phi_{\kappa(n)})(su(\mathbf{x},t)+(1-s)u(\sigma_{\alpha}(\mathbf{x}),t)) \,ds\right)\\&\times\left|\frac{P_tf(\mathbf{x})-P_tf(\sigma_{\alpha}(\mathbf{x}))}{\langle \mathbf{x},\alpha\rangle}\right|^2\,dt\,dw(\mathbf{x})   \\&\leq  \frac{2^{8}}{\gamma}\sum_{\alpha \in R}k(\alpha)\int_{\mathbb{R}^N}\Phi(\mathbf{x}/n)\int_0^{\infty}\nu_{\varepsilon}(t)\\&\times\left(\int_0^1 s\langle \Hess(s\widetilde{u}(\mathbf{x},t)+(1-s)\widetilde{u}(\mathbf{x},t))\rho_{\alpha}u(\mathbf{x},t),\rho_{\alpha}u(\mathbf{x},t) \rangle\,ds\right)\,dt\,dw(\mathbf{x}). 
\end{split}
\end{equation}
Now the claim is a direct consequence of~\eqref{eq:long_1},~\eqref{eq:long_2},~\eqref{eq:long_4}, and Proposition~\ref{propo:laplace_on_b}.\\
Finally, in order to prove~\eqref{eq:inv_version}, note that, by the definition of the Poisson semigroup (see Definition~\ref{def:Poisson_semigroup}) and Lemma~\ref{lem:poisson_properties}~\eqref{numitem:G_invariant}, for $G$-invariant $f$, the function $P_tf$ is also $G$-invariant for all $t>0$. Therefore, for all $1 \leq j \leq N$ we have
\begin{align*}
    T_jP_tf=\partial_{j}P_tf
\end{align*}
and the summand $I_2$ in~\eqref{eq:T_j_split} is equal to zero. Moreover, by~\eqref{eq:bell_2}, the third summand in~\eqref{eq:delta_on_bellman} is non-negative, so~\eqref{eq:inv_version} follows.
\end{proof}

\subsection{Upper estimate of \texorpdfstring{$I(n,\varepsilon,\kappa)$}{I(n,e,k)}}

\begin{lemma}\label{lem:part_3_1}
Let $p \geq 2$ and $q>1$ be such that $\frac{1}{p}+\frac{1}{q}=1$. Assume that $f,g_j \in \mathcal{S}(\mathbb{R}^N)$, $1 \leq j \leq N$, and $\varepsilon>0$. For $n \in \mathbb{N}$ we set
\begin{equation}\label{eq:kappa_n}
    \kappa(n)=\left(\frac{1}{n}\max(1,w(B(0,2n)))^{-1}\right)^{1/q}.
\end{equation}
Then we have
\begin{equation}\label{eq:zero}
    \limsup_{n \to \infty}\int_{\mathbb{R}^N}\Phi(\mathbf{x}/n)\int_0^{\infty}\nu_{\varepsilon}(t)\Delta_{k,\mathbf{x}}(b_{\kappa(n)})(\mathbf{x},t)\,dt\,dw(\mathbf{x})=0.
\end{equation}
\end{lemma}

\begin{proof}
Recall that $\supp \Phi \subseteq B(0,2)$. Therefore, by Lemma~\ref{lem:regularity}~\eqref{eq:numitem:smooth}, Corollary~\ref{coro:can_change_order}, and the fact that for fixed $\varepsilon>0$ we have $\int_0^{\infty} t^{-2}\nu_{\varepsilon}(t)\,dt<\infty$, we can change the order of integration in~\eqref{eq:zero}. Note that by the fact that $\int_0^{\infty}t^{-2}|\nu_{\varepsilon}(t)|\,dt<\infty$, it is enough to show that for fixed $t>0$ we have
\begin{align*}
    \limsup_{n \to \infty}\int_{\mathbb{R}^N}\Phi(\mathbf{x}/n)\Delta_{k,\mathbf{x}}(b_{\kappa(n)})(\mathbf{x},t)\,dw(\mathbf{x})=0.
\end{align*}
Recall that $b_{\kappa(n)} \in C^{\infty}(\mathbb{R}^N \times (0,\infty))$. Integrating by parts (see Theorem~\ref{teo:by_parts} and Remark~\ref{rem:by_parts}), for any $t>0$ we get
\begin{align*}
    \int_{\mathbb{R}^N}\Phi(\mathbf{x}/n)\Delta_{k,\mathbf{x}}(b_{\kappa(n)})(\mathbf{x},t)\,dw(\mathbf{x})=\int_{\mathbb{R}^N}\Delta_k\Phi(\mathbf{x}/n)(b_{\kappa(n)})(\mathbf{x},t)\,dw(\mathbf{x}).
\end{align*}
Recall that $\supp \Phi(\cdot/n) \subseteq B(0,2n)$. Then, it follows from Lemma~\ref{lem:bdd} that there is a constant $C>0$ independent of $\Phi$ and $n$ such that for all $\mathbf{x} \in \mathbb{R}^N$ and $n \in \mathbb{N}$ we have
\begin{equation}\label{eq:Phi_1}
    |(\Delta_k\Phi)(\mathbf{x}/n)| \leq C\sup_{\mathbf{y} \in \mathbb{R}^N}\sum_{\beta \in \mathbb{N}_0^N,\; |\beta| \leq 2}|\partial^{\beta}_{\mathbf{x}}\Phi(\mathbf{y}/n)| \leq C \sum_{\beta \in \mathbb{N}_0^N,\; |\beta| \leq 2}\|\partial^{\beta}\Phi\|_{L^{\infty}} \leq C_{\Phi}. 
\end{equation}
Moreover, by~\eqref{eq:laplace_formula} and the fact that $\Phi(\mathbf{x}/n)=1$ for all $\mathbf{x} \in B(0,n)$ we have 
\begin{align*}
    \Delta_k\Phi(\mathbf{x}/n)=0 \text{ for all }\mathbf{x} \in B(0,n).
\end{align*}
Consequently, by~\eqref{eq:bell_1} there is a constant $C_p>0$, which depends just on $p$, such that for all $n \in \mathbb{N}$ we have
\begin{equation}\label{eq:by_parts_app}
    \left|\int_{\mathbb{R}^N}(\Delta_k\Phi)(\mathbf{x}/n)(b_{\kappa(n)})(\mathbf{x},t)\,dw(\mathbf{x}) \right|\leq C_pC_{\Phi}\int_{B(0,2n) \setminus B(0,n)}|P_tf(\mathbf{x})|^{p}+\|P_t\mathbf{g}(\mathbf{x})\|^{q}+\kappa(n)^{q}\,dw(\mathbf{x}).
\end{equation}
Since $f,g_j \in \mathcal{S}(\mathbb{R}^N)$, by Lemma~\ref{lem:basic} we get $P_tf \in L^p(dw)$ and $P_tg_j \in L^q(dw)$ for all $t>0$ and $1 \leq j \leq N$. Hence,
\begin{align*}
    \lim_{n \to \infty}\int_{B(0,2n) \setminus B(0,n)}|P_tf(\mathbf{x})|^{p}+\|P_t\mathbf{g}(\mathbf{x})\|^{q}\,dw(\mathbf{x})=0.
\end{align*}
Moreover, by the choice of $\kappa(n)$ (see~\eqref{eq:kappa_n}) we get
\begin{align*}
    \lim_{n \to \infty}\int_{B(0,2n) \setminus B(0,n)} \kappa(n)^{q}\,dw(\mathbf{x})=\lim_{n \to \infty}\frac{1}{n}\frac{w(B(0,2n)\setminus B(0,n))}{w(B(0,2n))}=0.
\end{align*}
Therefore,
\begin{align*}
    \limsup_{n \to \infty}\int_{\mathbb{R}^N}\Phi(\mathbf{x}/n)\Delta_{k,\mathbf{x}}(b_{\kappa(n)})(\mathbf{x},t)\,dw(\mathbf{x})=\limsup_{n \to \infty}\int_{\mathbb{R}^N}\Delta_k\Phi(\mathbf{x}/n)(b_{\kappa(n)})(\mathbf{x},t)\,dw(\mathbf{x})=0.
\end{align*}
\end{proof}

\begin{lemma}\label{lem:boundary}
Assume that $f,g_j \in \mathcal{S}(\mathbb{R}^N)$, $1 \leq j \leq N$, $\kappa \in (0,1]$, and $\varepsilon>0$. Then for all $\mathbf{x} \in \mathbb{R}^N$ we have
\begin{equation}\label{eq:limits_1}
    \lim_{t \to 0}\nu_{\varepsilon}(t)|\partial_t b_{\kappa}(\mathbf{x},t)|=0, \ \ \ \lim_{t \to \infty}\nu_{\varepsilon}(t)|\partial_t b_{\kappa}(\mathbf{x},t)|=0,
\end{equation}
\begin{equation}\label{eq:limits_2}
    \lim_{t \to 0}|\nu_{\varepsilon}'(t)||b_{\kappa}(\mathbf{x},t)|=0, \ \ \ \lim_{t \to \infty}|\nu_{\varepsilon}'(t)|| b_{\kappa}(\mathbf{x},t)|=0.
\end{equation}
Recall that $\nu_{\varepsilon}$ is defined in~\eqref{eq:nu}.
\end{lemma}

\begin{proof}
Note that~\eqref{eq:limits_1} is a consequence of Lemma~\ref{lem:regularity}~\eqref{numitem:t_growth} and the fact that for fixed $\varepsilon>0$ we have
\begin{align*}
    \lim_{t \to 0}\frac{1}{t}\nu_{\varepsilon}(t)=\lim_{t \to \infty}\frac{1}{t}\nu_{\varepsilon}(t)=0.
\end{align*}
The proof of~\eqref{eq:limits_2} is similar. Indeed, since $f,g_j \in \mathcal{S}(\mathbb{R}^N)$, by~\eqref{DtDxDyPoisson} there is a constant $C=C_{f,\mathbf{g}}>0$ such that for all $\mathbf{x} \in \mathbb{R}^N$ and $t>0$ we have
\begin{align*}
    |P_tf(\mathbf{x})| \leq C \text{ and }|P_tg_j(\mathbf{x})| \leq C.
\end{align*}
Consequently, by~\eqref{eq:bell_1}, there is a constant $C'>0$, which depends on $f$ and $g_j$ and is independent of $\kappa \in (0,1]$, such that for all $\mathbf{x} \in \mathbb{R}^N$ and $t>0$ we have
\begin{align*}
    0 \leq b_{\kappa}(\mathbf{x},t) \leq C',
\end{align*}
so the claim is a consequence of an elementary fact that for fixed $\varepsilon>0$ we have
\begin{align*}
    \lim_{t \to 0}|\nu_{\varepsilon}'(t)|=\lim_{t \to \infty}|\nu_{\varepsilon}'(t)|=0.
\end{align*}
\end{proof}

\begin{lemma}\label{lem:nu_int}
Recall that $\nu_{\varepsilon}$ is defined in~\eqref{eq:nu}. We have
\begin{align*}
    \limsup_{\varepsilon \to 0^{+}}\int_{0}^{\infty}|\nu_{\varepsilon}''(t)|\,dt \leq 2(1+e^{-2}).
\end{align*}
\end{lemma}

\begin{proof}
It follows from an elementary calculation (see e.g.~\cite[(3.30)]{Sch}).
\end{proof}

\begin{lemma}\label{lem:part_3_2}
Let $p \geq 2$ and $q>1$ be such that $\frac{1}{p}+\frac{1}{q}=1$. Assume that $f,g_j \in \mathcal{S}(\mathbb{R}^N)$, $1 \leq j \leq N$, and $\varepsilon>0$. For $n \in \mathbb{N}$ we set
\begin{equation}
    \kappa(n)=\left(\frac{1}{n}\max(1,w(B(0,2n)))^{-1}\right)^{1/q}.
\end{equation}
Then we have
\begin{equation}\label{eq:nonzero}
    \limsup_{\varepsilon \to 0^{+}}\limsup_{n \to \infty}\int_{\mathbb{R}^N}\Phi(\mathbf{x}/n)\int_0^{\infty}\nu_{\varepsilon}(t)\partial_t^2(b_{\kappa(n)})(\mathbf{x},t)\,dt\,dw(\mathbf{x}) \leq 3(1+\gamma)\big(\|f\|_{L^p(dw)}^p+\|\mathbf{g}\|_{L^q(dw)}^q\big).
\end{equation}
\end{lemma}

\begin{proof}
Integrating by parts with respect to $t$ without boundary terms (it is possible thanks to Lemma~\ref{lem:boundary}) we get
\begin{equation}\label{eq:mid_0}
    \int_{\mathbb{R}^N}\Phi(\mathbf{x}/n)\int_0^{\infty}\nu_{\varepsilon}(t)\partial_t^2(b_{\kappa(n)})(\mathbf{x},t)\,dt\,dw(\mathbf{x})=\int_{\mathbb{R}^N}\Phi(\mathbf{x}/n)\int_0^{\infty}\nu_{\varepsilon}''(t)(b_{\kappa(n)})(\mathbf{x},t)\,dt\,dw(\mathbf{x}).
\end{equation}
Then, by~\eqref{eq:bell_1}, we have
\begin{align*}
    &\left|\int_{\mathbb{R}^N}\Phi(\mathbf{x}/n)\int_0^{\infty}\nu_{\varepsilon}''(t)(b_{\kappa(n)})(\mathbf{x},t)\,dt\,dw(\mathbf{x})\right| \\&\leq {(1+\gamma)}\int_{\mathbb{R}^N}\Phi(\mathbf{x}/n)\int_0^{\infty}|\nu_{\varepsilon}''(t)|\big((|P_tf(\mathbf{x})|+\kappa(n))^{p}+(\|P_t\mathbf{g}(\mathbf{x})\|+\kappa(n))^{q}\big)\,dt\,dw(\mathbf{x}).
\end{align*}
For fixed $t>0$ let $A_t:=\{\mathbf{x} \in \mathbb{R}^N\;:\;\varepsilon|P_tf(\mathbf{x})| \geq \kappa(n) \}$. Then
\begin{align*}
    &\int_{\mathbb{R}^N}\Phi(\mathbf{x}/n)\int_0^{\infty}|\nu_{\varepsilon}''(t)|((|P_tf(\mathbf{x})|+\kappa(n))^{p})\,dt\,dw(\mathbf{x})=\int_0^{\infty}\int_{A_t} \ldots+\int_0^{\infty}\int_{A_t^{c}}\ldots\\&\leq (1+\varepsilon)^{p}\int_{\mathbb{R}^N}\Phi(\mathbf{x}/n)\int_0^{\infty}|\nu_{\varepsilon}''(t)||P_tf(\mathbf{x})|^p\,dt\,dw(\mathbf{x})\\&+(1+\varepsilon^{-1})^p\int_{\mathbb{R}^N}\Phi(\mathbf{x}/n)\int_0^{\infty}|\nu_{\varepsilon}''(t)||\kappa(n)|^p\,dt\,dw(\mathbf{x}).
\end{align*}
Recall that $\supp \Phi(\cdot /n) \subseteq B(0,2n)$ and $0 \leq \Phi(\cdot /n) \leq 1$ for all $n \in \mathbb{N}$. Consequently, by the choice of $\kappa(n)$ (see~\eqref{eq:kappa_n}) we get
\begin{align*}
    &\limsup_{n \to \infty}(1+\varepsilon^{-1})^p\int_{\mathbb{R}^N}\Phi(\mathbf{x}/n)\int_0^{\infty}|\nu_{\varepsilon}''(t)||\kappa(n)|^p\,dt\,dw(\mathbf{x}) \\&\leq \lim_{n \to \infty} \left(\int_0^{\infty}|\nu_{\varepsilon}''(t)|\,dt\right)(1+\varepsilon^{-1})^p w(B(0,2n))\frac{1}{n^{p/q}w(B(0,2n))^{p/q}}=0.
\end{align*}
Therefore,
\begin{equation}\label{eq:mid_1}
\begin{split}
    &\limsup_{n \to \infty}\int_{\mathbb{R}^N}\Phi(\mathbf{x}/n)\int_0^{\infty}|\nu_{\varepsilon}''(t)|(|P_tf(\mathbf{x})|+\kappa(n))^{p}\,dt\,dw(\mathbf{x})\\&=(1+\varepsilon)^{p}\limsup_{n \to \infty} \int_{\mathbb{R}^N}\Phi(\mathbf{x}/n)\int_0^{\infty}|\nu_{\varepsilon}''(t)||P_tf(\mathbf{x})|^p\,dt\,dw(\mathbf{x}).
\end{split}
\end{equation}
Similarly,
\begin{equation}\label{eq:mid_2}
\begin{split}
    &\limsup_{n \to \infty}\int_{\mathbb{R}^N}\Phi(\mathbf{x}/n)\int_0^{\infty}|\nu_{\varepsilon}''(t)|(\|P_t\mathbf{g}(\mathbf{x})\|+\kappa(n))^{q}\,dt\,dw(\mathbf{x})\\&=\limsup_{n \to \infty} (1+\varepsilon)^{q}\int_{\mathbb{R}^N}\Phi(\mathbf{x}/n)\int_0^{\infty}|\nu_{\varepsilon}''(t)|\|P_t\mathbf{g}(\mathbf{x})\|^q\,dt\,dw(\mathbf{x}).
\end{split}
\end{equation}
Note that by H\"older's inequality and Lemma~\ref{lem:poisson_properties}~\eqref{numitem:poisson_integral_one} and~\eqref{numitem:poisson_positive} for all $\mathbf{x} \in \mathbb{R}^N$ and $t>0$ we have
\begin{align*}
    |P_tf(\mathbf{x})|^{p} \leq P_t(|f(\cdot)|^p)(\mathbf{x}), \ \ \ \|P_t\mathbf{g}(\mathbf{x})\|^{q} \leq P_t(\|\mathbf{g}(\cdot)\|^q)(\mathbf{x}).
\end{align*}
Furthermore, by Lemma~\ref{lem:poisson_properties} we have
\begin{align*}
    \int_{\mathbb{R}^N} |P_tf(\mathbf{x})|^{p} \,dw(\mathbf{x}) \leq \|f\|_{L^p(dw)}^p, \ \ \ \int_{\mathbb{R}^N}P_t(\|\mathbf{g}(\cdot)\|^q)(\mathbf{x})\,dw(\mathbf{x}) \leq \|\mathbf{g}\|_{L^q(dw)}^q.
\end{align*}
Consequently, by the Fubini theorem, the fact that $0 \leq \Phi(\cdot/n) \leq 1$, and Lemma~\ref{lem:regularity}, we get
\begin{equation}\label{eq:mid_3}
\begin{split}
    &\limsup_{n \to \infty} \int_{\mathbb{R}^N}\Phi(\mathbf{x}/n)\int_0^{\infty}|\nu_{\varepsilon}''(t)|(|P_tf(\mathbf{x})|^p+\|P_t\mathbf{g}(\mathbf{x})\|^q)\,dt\,dw(\mathbf{x})\\& \leq \limsup_{n \to \infty} \int_0^{\infty}|\nu_{\varepsilon}''(t)|\int_{\mathbb{R}^N}(|P_tf(\mathbf{x})|^p+\|P_t\mathbf{g}(\mathbf{x})\|^q)\,dw(\mathbf{x})\,dt \\& \leq \left(\int_0^{\infty}|\nu_{\varepsilon}''(t)|\,dt\right) \left(\|f\|_{L^p(dw)}^{p}+\|\mathbf{g}\|_{L^q(dw)}^q\right).
\end{split}
\end{equation}
Finally, by~\eqref{eq:mid_0},~\eqref{eq:mid_1},~\eqref{eq:mid_2},~\eqref{eq:mid_3}, and Lemma~\ref{lem:nu_int} we obtain 
\begin{align*}
    &\limsup_{\varepsilon \to 0^{+}}\limsup_{n \to \infty}\int_{\mathbb{R}^N}\Phi(\mathbf{x}/n)\int_0^{\infty}\nu_{\varepsilon}(t)\partial_t^2(b_{\kappa(n)})(\mathbf{x},t)\,dt\,dw(\mathbf{x}) \\&\leq (1+\gamma)2(1+e^{-2})\big(\|f\|_{L^p}^{p}+\|\mathbf{g}\|_{L^q(dw)}^q\big) \leq 3(1+\gamma)\left(\|f\|_{L^p(dw)}^{p}+\|\mathbf{g}\|_{L^q(dw)}^q\right).
\end{align*}
\end{proof}

As a direct consequence of Lemmas~\ref{lem:part_3_1} and~\ref{lem:part_3_2} we obtain the following corollary.

\begin{corollary}\label{coro:part_3}
Let $p \geq 2$ and $q>1$ be such that $\frac{1}{p}+\frac{1}{q}=1$. Assume that $f,g_j \in \mathcal{S}(\mathbb{R}^N)$, $1 \leq j \leq N$. For $n \in \mathbb{N}$ we set
\begin{equation}
    \kappa(n)=\left(\frac{1}{n}\max(1,w(B(0,2n)))^{-1}\right)^{1/q}.
\end{equation}
Then we have
\begin{align*}
    \liminf_{\varepsilon \to 0^{+}}\liminf_{n \to \infty}I(n,\varepsilon,\kappa) \leq 3(1+\gamma)\left(\|f\|_{L^p(dw)}^{p}+\|\mathbf{g}\|_{L^q(dw)}^q\right),
\end{align*}
where $I(n,\varepsilon,\kappa)$ is defined in~\eqref{eq:I}. 
\end{corollary}

\subsection{Proof of Theorem~\ref{teo:main}}

\begin{proof}[Proof of Theorem~\ref{teo:main}]
We will prove~\eqref{eq:m_1} first. Assume first that $p \geq 2$. Take $f \in L^p(dw)$. Thanks to Theorem~\ref{teo:Amri} and the fact that $\mathcal{S}(\mathbb{R}^N)$ is dense in $L^p(dw)$, without loss of generality we can assume $f \in \mathcal{S}(\mathbb{R}^N)$. Let $\kappa:\mathbb{N} \to (0,1]$ be defined by~\eqref{eq:kappa_n}. By Corollary~\ref{coro:part_1} we get
\begin{equation}\label{eq:combine_0}
\begin{split}
    \|\mathcal{R}f\|_{L^p(dw)}&=4\sup_{g_j \in \mathcal{S}(\mathbb{R}^N),\; \left\|\|\mathbf{g}(\mathbf{y})\|\right\|_{L^q(dw(\mathbf{y}))} \leq 1}\left|\sum_{j=1}^{N}\int_{\mathbb{R}^N}\int_0^{\infty}t \partial_tP_tg_j(\mathbf{x})T_{j}P_tf(\mathbf{x})\,dt\,dw(\mathbf{x})\right|.
\end{split}
\end{equation}
Next, by Lemma~\ref{lem:part_2} and Corollary~\ref{coro:part_3},
\begin{equation}\label{eq:combine}
\begin{split}
    &4\sup_{g_j \in \mathcal{S}(\mathbb{R}^N),\; \left\|\|\mathbf{g}(\mathbf{y})\|\right\|_{L^q(dw(\mathbf{y}))} \leq 1}\left|\sum_{j=1}^{N}\int_{\mathbb{R}^N}\int_0^{\infty}t \partial_tP_tg_j(\mathbf{x})T_{j}P_tf(\mathbf{x})\,dt\,dw(\mathbf{x})\right| \\&\leq
    \frac{8}{\gamma}\left(\sum_{\alpha \in R}k(\alpha)+2^7\right)\liminf_{\varepsilon \to 0^{+}}\liminf_{n \to \infty}I(n,\varepsilon,\kappa)\\&\leq \frac{24(1+\gamma)}{\gamma}\left(\sum_{\alpha \in R}k(\alpha)+2^7\right)\left(\|f\|_{L^p(dw)}^{p}+\|\mathbf{g}\|_{L^q(dw)}^q\right).
\end{split}
\end{equation}
Finally, we use a polarization arguments. Let $s>0$. We replace $f(\cdot)$ by $sf(\cdot)$ and $\mathbf{g}(\cdot)$ by $s^{-1}\mathbf{g}(\cdot)$ in~\eqref{eq:combine}. Then, the left hand side of~\eqref{eq:combine} is unchanged, and minimizing the right-hand-side by $s>0$ we obtain 
\begin{align*}
     \|\mathcal{R}f\|_{L^p(dw)} \leq \frac{24(1+\gamma)}{\gamma}((p/q)^{1/p}+(q/p)^{1/q})\left(\sum_{\alpha \in R}k(\alpha)+2^7\right)\|f\|_{L^p(dw)}\|\|\mathbf{g}(\mathbf{y})\|\|_{L^q(dw(\mathbf{y}))}.
\end{align*}
It was shown in~\cite[proof of the main theorem]{Wrobel} that
\begin{align*}
    \frac{(1+\gamma)}{\gamma}((p/q)^{1/p}+(q/p)^{1/q}) \leq 6(p^{*}-1),
\end{align*}
which ends the proof for $p \geq 2$. The proof in case $1<p<2$ is analogous: we switch $P_tf$ and $P_t\mathbf{g}$ in the definition of $b_{\kappa}$. The proof of~\eqref{eq:m_2} is similar (we use~\eqref{eq:inv_version} instead of~\eqref{eq:noninv_version} in~\eqref{eq:combine}).
\end{proof}

\section{One-dimensional case}

This section is devoted to the proof of Theorem~\ref{teo:main2}. We will work in the one-dimensional setting, i.e. we assume $N=1$. We would like to emphasize that in this case we have
\begin{align*}
    R=\{\sqrt{2},-\sqrt{2}\}, \ \ \ G=\{{\rm id},\sigma_{-\sqrt{2}}\},
\end{align*}
where $\sigma_{-\sqrt{2}}(x)=-x$ for all $x \in \mathbb{R}$. Consequently, the multiplicity function $k$ takes just one value, which, for simplicity of the notation, will be denoted by $k$. In this case, the associated measure $dw$ is of the form
\begin{equation}\label{eq:measure_w_one}
    dw(x)=2|x|^{2k}\,dx.
\end{equation}
The Dunkl operator in one-dimensional case is
\begin{equation}\label{eq:T}
    Tf(x):=\partial_{x} f(x)+k\frac{f(x)-f(-x)}{x}.
\end{equation}
In this section, we will use the same notation as in the previous sections unless specified otherwise. We will also assume $k>1$ (otherwise, the claim follows by Theorem~\ref{teo:main}).
\\
We will slightly modify the proof of Theorem~\ref{teo:main} to obtain the Theorem~\ref{teo:main2}. The main point is to prove a modified version of Lemma~\ref{lem:part_2}. We will also need the following version of Proposition~\ref{propo:dual_1}. We state and prove it for the convenience of the reader.

\begin{proposition}\label{propo:dual_1_one}
For all $f,g \in \mathcal{S}(\mathbb{R})$ we have
\begin{equation}
    \left|\int_{\mathbb{R}}\mathcal{H}f(x)g(x)\,dw(x)\right|=4\left|\int_0^{\infty}\int_{\mathbb{R}}t\partial_tP_tf(x)TP_tg(x)\,dw(x)\,dt\right|.
\end{equation}
\end{proposition}

\begin{proof}
By Plancherel's identity (see~\eqref{eq:Plancherel}) and the definition of the Dunkl Hilbert transform (see~\eqref{eq:Hilbert}) we have
\begin{align*}
    &\int_{\mathbb{R}}\mathcal{H}f(x)g(x)\,dw(x)=\int_{\mathbb{R}}\left(\frac{-i\xi}{|\xi|}\mathcal{F}f(\xi)\right)\mathcal{F}g(\xi)\,dw(\xi)\\&=\int_{\mathbb{R}}\mathcal{F}f(\xi)\left(\frac{-i\xi}{|\xi|}\mathcal{F}g(\xi)\right)\,dw(\xi)=-\int_{\mathbb{R}}\mathcal{H}g(x)f(x)\,dw(x).
\end{align*}
Consequently, the rest of the proof is the same as in the proof on Proposition~\ref{propo:dual_1} (with $g$ instead of $f$ and $f$ instead of $g$).
\end{proof}

Now we are ready to state and prove the modified version of Lemma~\ref{lem:part_2}. Recall that $I(n,\varepsilon,\kappa)$ is defined in Subsection~\ref{sec:I} (see~\eqref{eq:II}).

\begin{lemma}\label{lem:part2_modified}
Assume that $p \geq 2$, $q>1$ are such that $\frac{1}{p}+\frac{1}{q}=1$, and $f ,g \in \mathcal{S}(\mathbb{R})$. Assume that $g$ is odd. Then for any $\kappa:\mathbb{N} \to (0,1]$ we have
\begin{equation}\label{eq:part2_modified}
\begin{split}
     \int_{\mathbb{R}}\int_0^{\infty}t|\partial_tP_tf(x)||TP_tg(x)|\,dt\,dw(x) &\leq \frac{8}{\gamma}\liminf_{\varepsilon \to 0^{+}}\liminf_{n \to \infty}I(n,\varepsilon,\kappa)\\&+ \liminf_{\varepsilon \to 0^{+}}\liminf_{n \to \infty}e_1(n,\varepsilon,\kappa)\\&+\liminf_{\varepsilon \to 0^{+}}\liminf_{n \to \infty}e_2(n,\varepsilon,\kappa),
\end{split}
\end{equation}
where
\begin{equation}\label{eq:e1}
    e_1(n,\varepsilon,\kappa):=6\kappa(n)^{2-q} \int_{\mathbb{R}}\Phi(x/n)\int_0^{\infty}\nu_{\varepsilon}(t)|\partial_tP_tf(x)|^2\,dt\,dw(x),
\end{equation}
\begin{equation}\label{eq:e2}
    e_2(n,\varepsilon,\kappa):=-n^{-1}\int_{\mathbb{R}}(\partial_{x}\Phi)(x/n)\int_0^{\infty}\frac{4k}{q}(|P_tg(x)|^2+\kappa(n)^2)^{q/2}x^{-1}\,dt\,dw(x).
\end{equation}
\end{lemma}

\begin{remark}\normalfont
Let us note that in Lemma~\ref{lem:part2_modified} we do not have the factor "$k$" in front of "$\liminf_{\varepsilon \to 0^{+}}\liminf_{n \to \infty}I(n,\varepsilon,\kappa)$", which appears in Lemma~\ref{lem:part_2}. This is a crucial difference. However, there are two additional error terms $e_1(n,\varepsilon,\kappa)$ and $e_2(n,\varepsilon,\kappa)$, but we will show that they are negligible (see Lemmas~\ref{lem:e1} and~\ref{lem:e2}). The terms $e_1(n,\varepsilon,\kappa)$ and $e_2(n,\varepsilon,\kappa)$ appear in~\eqref{eq:part2_modified} just for technical reasons.
\end{remark}

\begin{proof}[Proof of Lemma~\ref{lem:part2_modified}]
Fix $\kappa:\mathbb{N} \to (0,1]$. By the monotone convergence theorem we have
\begin{align*}
    &\int_0^{\infty}\int_{\mathbb{R}}|t\partial_tP_tf(x)||TP_tg(x)|\,dw(x)\,dt \leq \\&\lim_{\varepsilon \to 0^{+}}\lim_{n \to \infty}\int_{\mathbb{R}}\Phi(x/n)\int_0^{\infty}\nu_{\varepsilon}(t)|\partial_tP_tf(x)||TP_tg(x)|\,dt\,dw(x).
\end{align*}
Recall that $\Phi$ and $\nu_{\varepsilon}$ are defined in Definition~\ref{def:Phi}. Fix $n \in \mathbb{N}$ and $\varepsilon>0$. It follows by the definition of the Poisson semigroup (see Definition~\ref{def:Poisson_semigroup}) and Lemma~\ref{lem:poisson_properties}~\eqref{numitem:G_invariant} that if $g$ is odd, then $P_tg$ is also odd for all $t>0$. Consequently, by~\eqref{eq:T},
\begin{align*}
    TP_tg(x)=\partial_x P_tg(x)+2k\frac{P_tg(x)}{x}.
\end{align*}
For $(y_1,y_2) \in \mathbb{R} \times \mathbb{R}$ we set
\begin{equation}\label{eq:tau_2}
    \tau_2(y_1,y_2)=\left(|y_2|^2+\kappa(n)^2\right)^{(2-q)/2}
\end{equation}
(cf.~\eqref{eq:tau_1}). By the inequality between the geometric and arithmetic mean we get 
\begin{align*}
    &\int_{\mathbb{R}}\Phi(x/n)\int_0^{\infty}\nu_{\varepsilon}(t)|\partial_tP_tf(x)|\left|\partial_{x} P_tg(x)+2k\frac{P_tg(x)}{x}\right|\,dt\,dw(x) \\& \leq \int_{\mathbb{R}}\Phi(x/n)\int_0^{\infty}\nu_{\varepsilon}(t)|\partial_tP_tf(x)|^2\tau_2(u(x,t))\,dt\,dw(x)\\&+\int_{\mathbb{R}}\Phi(x/n)\int_0^{\infty}\nu_{\varepsilon}(t)\left|\partial_xP_tg(x)+2k\frac{P_tg(x)}{x}\right|^2\tau_2^{-1}(u(x,t))\,dt\,dw(x)\\&=:I_1+I_2.
\end{align*}
We will estimate $I_1$ and $I_2$ separately.
\\
\noindent
\textbf{Estimate of $I_1$.} Recall that $\int_{\mathbb{R} \times \mathbb{R}}\phi_{\kappa(n)}(y_1,y_2)\,dy_1\,dy_2=1$ and $\supp \phi_{\kappa(n)} \subseteq B(0,\kappa(n))$ ($\phi_{\kappa(n)}$ is defined in~\eqref{eq:phi}). Moreover, $q \in (1,2]$, so $(2-q)/2 \geq 0$. Hence,
\begin{align*}
    \tau_2(u(x,t))&=\int_{\mathbb{R} \times \mathbb{R}}\phi_{\kappa(n)}(y_1,y_2)\tau_2(u(x,t))\,dy_1\,dy_2\\&=\int_{\mathbb{R} \times \mathbb{R}}\phi_{\kappa(n)}(y_1,y_2)(|P_tg(x)|^2+\kappa(n)^2)^{(2-q)/2}\,dy_1\,dy_2\\& \leq \int_{\mathbb{R} \times \mathbb{R}}\phi_{\kappa(n)}(y_1,y_2)(2|P_tg(x)-y_2|^2+2|y_2|^2+\kappa(n)^2)^{(2-q)/2}\,dy_1\,dy_2\\&\leq \int_{\mathbb{R} \times \mathbb{R}}\phi_{\kappa(n)}(y_1,y_2)(2|P_tg(x)-y_2|^2+3\kappa(n)^2)^{(2-q)/2}\,dy_1\,dy_2 \\&\leq \int_{\mathbb{R} \times \mathbb{R}}\phi_{\kappa(n)}(y_1,y_2)\left(4|P_tg(x)-y_2|^{(2-q)}+6\kappa(n)^{(2-q)}\right)\,dy_1\,dy_2\\&=4\phi_{\kappa(n)} \star \tau_1(u(x,t))+6\kappa(n)^{2-q}, 
\end{align*}
where $\tau_1$ and $u(x,t)$ are defined in~\eqref{eq:tau_1} and~\eqref{eq:u} respectively. Therefore,
\begin{equation}\label{eq:one_main_lemma_I1}
\begin{split}
    I_1 &\leq 4\int_{\mathbb{R}}\Phi(x/n)\int_0^{\infty}\nu_{\varepsilon}(t)(\tau_1\star\phi_{\kappa(n)})(u(x,t))|\partial_tP_tf(x)|^2\,dt\,dw(x)\\&+6\kappa(n)^{2-q} \int_{\mathbb{R}}\Phi(x/n)\int_0^{\infty}\nu_{\varepsilon}(t)|\partial_tP_tf(x)|^2\,dt\,dw(x)\\&=4\int_{\mathbb{R}}\Phi(x/n)\int_0^{\infty}\nu_{\varepsilon}(t)(\tau_1\star\phi_{\kappa(n)})(u(x,t))|\partial_tP_tf(x)|^2\,dt\,dw(x)+e_1(n,\varepsilon,\kappa).
\end{split}
\end{equation}
\\
\noindent
\textbf{Estimate of $I_2$.} We split $I_2$ into three parts:
\begin{equation}\label{eq:expand}
\begin{split}
    &\int_{\mathbb{R}}\Phi(x/n)\int_0^{\infty}\nu_{\varepsilon}(t)\left|\partial_xP_tg(x)+2k\frac{P_tg(x)}{x}\right|^2\tau_2^{-1}(u(x,t))\,dt\,dw(x)\\&=\int_{\mathbb{R}}\Phi(x/n)\int_0^{\infty}\nu_{\varepsilon}(t)|\partial_xP_tg(x)|^2\tau_2^{-1}(u(x,t))\,dt\,dw(x)\\&+\int_{\mathbb{R}}\Phi(x/n)\int_0^{\infty}\nu_{\varepsilon}(t)4k^2\frac{|P_tg(x)|^2}{x^2}\tau_2^{-1}(u(x,t))\,dt\,dw(x)\\&+\int_{\mathbb{R}}\Phi(x/n)\int_0^{\infty}\nu_{\varepsilon}(t)4k(\partial_x P_tg)(x)\frac{P_tg(x)}{x}\tau_2^{-1}(u(x,t))\,dt\,dw(x)=:J_{1}+J_{2}+J_{3}.
\end{split}
\end{equation}
We will estimate $J_1$ and $J_2+J_3$ separately.
\\
\noindent
\textbf{Estimate of $J_1$.} Recall that $\int_{\mathbb{R} \times \mathbb{R}}\phi_{\kappa(n)}(y_1,y_2)\,dy_1\,dy_2=1$ and $\supp \phi_{\kappa(n)} \subseteq B(0,\kappa(n))$. Therefore, by the definitions of $\tau_2$ and $\tau_1$ (see~\eqref{eq:tau_2} and~\eqref{eq:tau_1} respectively), the fact that $q \in (1,2]$, and the triangle inequality we get
\begin{equation}\label{eq:I_21}
\begin{split}
    &\int_{\mathbb{R}}\Phi(x/n)\int_0^{\infty}\nu_{\varepsilon}(t)|\partial_xP_tg(x)|^2\tau_{2}^{-1}(u(x,t))\,dt\,dw(x)\\&=\int_{\mathbb{R}}\Phi(x/n)\int_0^{\infty}\nu_{\varepsilon}(t)|\partial_xP_tg(x)|^2\left(\int_{\mathbb{R} \times \mathbb{R}}\phi_{\kappa(n)}(y_1,y_2)\tau_{2}^{-1}(u(x,t))\,dy_1\,dy_2\right)\,dt\,dw(x)\\&\leq 2\int_{\mathbb{R}}\Phi(x/n)\int_0^{\infty}\nu_{\varepsilon}(t)|\partial_xP_tg(x)|^2\left(\int_{\mathbb{R} \times \mathbb{R}}\phi_{\kappa(n)}(y_1,y_2)\tau_{1}^{-1}(u(x,t)-(y_1,y_2))\,dy_1\,dy_2\right)\,dt\,dw(x)\\&= 2\int_{\mathbb{R}}\Phi(x/n)\int_0^{\infty}\nu_{\varepsilon}(t)(\tau_{1}^{-1}\star\phi_{\kappa(n)})(u(x,t))|\partial_xP_tg(x)|^2\,dw(x)\,dt.
\end{split}
\end{equation}
\\
\noindent
\textbf{Estimate of $J_2+J_3$.} By the definition of $\tau_2$ (see~\eqref{eq:tau_2}), we get
\begin{align*}
    4k(\partial_x P_tg)(x)\frac{P_tg(x)}{x}\tau_2^{-1}(u(x,t))=x^{-1}\frac{4k}{q}\partial_{x}((|P_tg(x)|^2+\kappa(n)^2)^{q/2}).
\end{align*}
Therefore, by~\eqref{eq:measure_w_one} and the integration by parts (with respect to the Lebesgue measure), for any $t>0$ we get
\begin{equation}\label{eq:lebesque_by_parts}
\begin{split}
    &J_3=\int_{\mathbb{R}}\Phi(x/n)4k(\partial_x P_tg)(x)\frac{P_tg(x)}{x}\tau_2^{-1}(u(x,t))\,dw(x)\\&=\int_{\mathbb{R}}\frac{4k}{q}\partial_{x}((|P_tg(x)|^2+\kappa(n)^2)^{q/2})\left(x^{-1}2|x|^{2k}\Phi(x/n)\right)\,dx\\&=-\int_{\mathbb{R}}(|P_tg(x)|^2+\kappa(n)^2)^{q/2}\frac{4k(2k-1)}{q}\frac{2|x|^{2k}}{x^2}\Phi(x/n)\,dx\\&-\int_{\mathbb{R}}(|P_tg(x)|^2+\kappa(n)^2)^{q/2}\frac{4k}{q}\frac{2|x|^{2k}}{x}\left(n^{-1}(\partial_x\Phi)(x/n)\right)\,dx=J_{3,1}+e_2(n,\varepsilon,\kappa).
\end{split}
\end{equation}
Recall that $q \in (1,2]$, so 
\begin{align*}
    0 \leq 4k^2-\frac{4k(2k-1)}{q} \leq \frac{4k}{q}.
\end{align*}
Hence, by the definitions of $J_2$ and $J_{3,1}$ and the assumption $k>1$, we get
\begin{equation}\label{eq:crucial}
\begin{split}
    &J_{2}+J_{3,1}=\\&\int_{\mathbb{R}}\Phi(x/n)\int_{0}^{\infty}\nu_{\varepsilon}(t)\left(4k^2|P_tg(x)|^2-\frac{4k(2k-1)}{q}(|P_tg(x)|^2+\kappa(n)^2)\right)\frac{\tau_2^{-1}(u(x,t))}{x^2}\,dt\,dw(x)\\& \leq \int_{\mathbb{R}}\Phi(x/n)\int_{0}^{\infty}\nu_{\varepsilon}(t)\frac{4k}{q}|P_tg(x)|^2\frac{\tau_2^{-1}(u(x,t))}{x^2}\,dt\,dw(x).
\end{split}
\end{equation}
Note that, by the fact that $q \in (1,2]$ and the triangle inequality, for all $(y_1,y_2) \in B(0,\kappa(n))$ we have
\begin{equation}\label{eq:tau_2_comp}
    \tau_2^{-1}(u(x,t)) \leq 2\max(|P_tg(x)-y_2|,|-P_tg(x)-y_2|)^{q-2}.
\end{equation}
Recall that $P_tg$ is odd for all $t>0$, so 
\begin{align*}
    \frac{2P_tg(x)}{x}=\frac{P_tg(x)-P_{t}g(-x)}{x}.
\end{align*}
Therefore, by~\eqref{eq:tau_2_comp} and~\eqref{eq:elem_2} with $\mathbf{a}=P_tg(x)-y_2$ and $\mathbf{b}=-P_tg(x)-y_2=P_tg(-x)-y_2$, we get
\begin{equation}\label{eq:I22}
\begin{split}
    &\int_{\mathbb{R}}\Phi(x/n)\int_{0}^{\infty}\nu_{\varepsilon}(t)\frac{4k}{q}|P_tg(x)|^2\frac{\tau_2^{-1}(u(x,t))}{x^2}\,dt\,dw(x)\\&=\int_{\mathbb{R}}\Phi(x/n)\int_{0}^{\infty}\nu_{\varepsilon}(t)\frac{4k}{q}\left(\int_{\mathbb{R} \times \mathbb{R}}\phi_{\kappa(n)}(y_1,y_2)\tau_2^{-1}(u(x,t))\,dy_1\,dy_2\right)\frac{|P_tg(x)|^2}{x^2}\,dt\,dw(x)\\&\leq 2\int_{\mathbb{R}}\Phi(x/n)\int_{0}^{\infty}\nu_{\varepsilon}(t)\frac{4k}{q}\left(\int_{\mathbb{R} \times \mathbb{R}}\phi_{\kappa(n)}(y_1,y_2)\max\big(|P_tg(x)-y_2|,|P_tg(-x)-y_2|\big)^{q-2}\,dy_1\,dy_2\right)\\&\times\frac{|P_tg(x)|^2}{x^2}\,dt\,dw(x)\\&\leq 4\int_{\mathbb{R}}\Phi(x/n)\int_{0}^{\infty}\nu_{\varepsilon}(t)\frac{4k}{q}\left(\int_{\mathbb{R} \times \mathbb{R}}\left(\int_0^1\phi_{\kappa(n)}(y_1,y_2)\tau_1^{-1}(su(x,t)+(1-s)u(-x,t)-y_2)\,ds\right)\,dy_1\,dy_2\right)\\&\times\frac{|P_tg(x)|^2}{x^2}\,dt\,dw(x)\\&=4\int_{\mathbb{R}}\Phi(x/n)\int_{0}^{\infty}\nu_{\varepsilon}(t)\frac{4k}{q}\left(\int_0^{1}s(\phi_{\kappa(n)}\star\tau_1^{-1})(su(x,t)+(1-s)u(-x,t))\,ds\right) \frac{|P_tg(x)|^2}{x^2}\,dt\,dw(x).
\end{split}
\end{equation}
Finally, by~\eqref{eq:expand},~\eqref{eq:I_21},~\eqref{eq:lebesque_by_parts}, and~\eqref{eq:I22}, 
\begin{equation}\label{eq:J_23}
\begin{split}
    I_2& \leq 2\int_{\mathbb{R}}\Phi(x/n)\int_0^{\infty}\nu_{\varepsilon}(t)(\phi_{\kappa(n)} \star\tau_{1}^{-1})(u(x,t))|\partial_xP_tg(x)|^2\,dw(x)\,dt\\&+4k\int_{\mathbb{R}}\Phi(x/n)\int_{0}^{\infty}\nu_{\varepsilon}(t)\left(\int_0^{1}s(\phi_{\kappa(n)}\star\tau_1^{-1})(su(x,t)+(1-s)u(-x,t))\,ds\right)\\&\times \frac{|P_tg(x)-P_tg(-x)|^2}{x^2}\,dt\,dw(x)\\&+e_2(n,\varepsilon,\kappa).
\end{split}
\end{equation}
Now we are ready to apply the same argument as in the proof on Lemma~\ref{lem:part_2}. Indeed, by~\eqref{eq:bell_2} with
\begin{align*}
\begin{cases}
    N_1=N_2=1,\\
    \eta=P_tf(x),\\
    \zeta=P_tg(x),\\
    \omega=\partial_{t}u(x,t)\text{ or }\omega=\partial_{x}u(x,t)
\end{cases}
\end{align*}
(see~\eqref{eq:u}) we get
\begin{equation}\label{eq:modified_final}
\begin{split}
    \\&4\int_{\mathbb{R}}\Phi(x/n)\int_0^{\infty}\nu_{\varepsilon}(t)(\tau_1\star\phi_{\kappa(n)})(u(x,t))|\partial_tP_tf(x)|^2\,dt\,dw(x)\\&+2\int_{\mathbb{R}}\Phi(x/n)\int_0^{\infty}\nu_{\varepsilon}(t)(\tau_{1}^{-1}\star\phi_{\kappa(n)})(u(x,t))|\partial_xP_tg(x)|^2\,dw(x)\,dt\\&\leq \frac{8}{\gamma}\int_{\mathbb{R}}\Phi(x/n)\int_0^{\infty}\nu_{\varepsilon}(t)\langle \Hess(B_{\kappa(n)})(\widetilde{u}(x,t))\partial_t u(x,t),\partial_t u(x,t)\rangle \,dt\,dw(x)\\&+\frac{8}{\gamma}\int_{\mathbb{R}}\Phi(x/n)\int_0^{\infty}\nu_{\varepsilon}(t)\langle \Hess(B_{\kappa(n)})(\widetilde{u}(x,t))\partial_x u(x,t),\partial_x u(x,t)\rangle \,dt\,dw(x).
\end{split}
\end{equation}
Then, by~\eqref{eq:bell_2} with
\begin{align*}
\begin{cases}
    N_1=N_2=1,\\
    \eta=sP_tf(x)+(1-s)P_tf(-x),\\
    \zeta=sP_tg(x)+(1-s)P_tg(-x),\\
    \omega=\rho_{-\sqrt{2}}u(x,t)=\left(\frac{P_tf(x)-P_tf(-x)}{x},\frac{P_tg(x)-P_tg(-x)}{x}\right),
\end{cases}
\end{align*}
where $s \in [0,1]$, we obtain
\begin{equation}\label{eq:modified_final_1}
\begin{split}
    &4k\int_{\mathbb{R}}\Phi(x/n)\int_{0}^{\infty}\nu_{\varepsilon}(t)\left(\int_0^{1}s(\phi_{\kappa(n)}\star\tau_1^{-1})(su(x,t)+(1-s)u(-x,t))\,ds\right)\\&\times \frac{|P_tg(x)-P_tg(-x)|^2}{x^2}\,dt\,dw(x)\\&\leq \frac{8k}{\gamma}\int_{\mathbb{R}}\Phi(x/n)\int_0^{\infty}\nu_{\varepsilon}(t)\\&\times\left(\int_0^1 s \langle \Hess(B_{\kappa(n)})(s\widetilde{u}(x,t)+(1-s)\widetilde{u}(-x,t))\rho_{-\sqrt{2}}u(x,t),\rho_{-\sqrt{2}}u(x,t)\rangle\,ds\right) \,dt\,dw(x).  
\end{split}
\end{equation}
Finally, the claim is a consequence of~\eqref{eq:one_main_lemma_I1},~\eqref{eq:J_23},~\eqref{eq:modified_final}~\eqref{eq:modified_final_1}, and Proposition~\ref{propo:laplace_on_b}.
\end{proof}

\begin{remark}\normalfont
We would like to emphasize that we get rid of "$k^2$" factor in~\eqref{eq:lebesque_by_parts} and~\eqref{eq:crucial}, which is an crucial difference between the proofs of Lemma~\ref{lem:part_2} and Lemma~\ref{lem:part2_modified}.
\end{remark}

For the sake of completeness, we also formulate an analogue of Lemma~\ref{lem:part2_modified} for even functions. Its proof is identical as the proof of Lemma~\ref{lem:part_2}.

\begin{lemma}\label{lem:part2_even}
Assume that $f ,g \in \mathcal{S}(\mathbb{R})$ and $g$ is even. Then for any $\kappa:\mathbb{N} \to (0,1]$ we have
\begin{equation}
\begin{split}
     \int_0^{\infty}\int_{\mathbb{R}}|t\partial_tP_tf(x)||TP_tg(x)|\,dw(x)\,dt &\leq \frac{2}{\gamma}\liminf_{\varepsilon \to 0^{+}}\liminf_{n \to \infty}I(n,\varepsilon,\kappa).
\end{split}
\end{equation}
\end{lemma}

\begin{proof}
Note that by the definition of Poisson semigroup (see Definition~\ref{def:Poisson_semigroup}) and Lemma~\ref{lem:poisson_properties}~\eqref{numitem:G_invariant}, if $g$ is even, then $P_tg$ is also even for all $t>0$. Hence, by~\eqref{eq:T}, for all $t>0$ and $x \in \mathbb{R}$ we get
\begin{align*}
    TP_tg(x)=\partial_xP_tg(x).
\end{align*}
Therefore, the rest part of the proof is the same as in the proof of Lemma~\ref{lem:part_2} since $I_2=0$ in~\eqref{eq:T_j_split}.
\end{proof}

\begin{lemma}\label{lem:e1}
Assume that $f \in \mathcal{S}(\mathbb{R})$, $\kappa:\mathbb{N} \to (0,1]$, $\varepsilon>0$, and $e_1(n,\varepsilon,\kappa)$ is defined in~\eqref{eq:e1}. If $\lim_{n \to \infty}\kappa(n)=0$, then
\begin{equation}
    \limsup_{n \to \infty}e_1(n,\varepsilon,\kappa)=0.
\end{equation}
\end{lemma}

\begin{proof}
If follows by the definition of the Poisson semigroup (see Definition~\ref{def:Poisson_semigroup}) and~\eqref{DtDxDyPoisson} that there is a constant $C>0$ independent of $f$ such that for all $t>0$ and $x \in \mathbb{R}$ we have
\begin{align*}
    |(\partial_tP_tf)(x)|^2 \leq Ct^{-2}P_t|f|(x)^2.
\end{align*}
Therefore, by the fact that $0 \leq \Phi \leq 1$ and, by definition, $\{P_t\}_{t \geq 0}$ are contractions on $L^2(dw)$, we get
\begin{equation}\label{eq:e1_lem}
\begin{split}
    |e_1(n,\varepsilon,\kappa)| &\leq 6C\kappa(n)^{2-q}\int_{0}^{\infty}\nu_{\varepsilon}(t)\left(\int_{\mathbb{R}}\Phi(x/n)t^{-2}|P_tf(x)|^2\,dw(x)\right)\,dt \\&\leq 6C\kappa(n)^{2-q}\|f\|^2_{L^2(dw)}\left(\int_{0}^{\infty}t^{-2}\nu_{\varepsilon}(t)\,dt\right).
\end{split}
\end{equation}
Recall that $f \in \mathcal{S}(\mathbb{R})$, so $\|f\|_{L^2(dw)}<\infty$. Moreover, by the definition of $\nu_{\varepsilon}$ (see~\eqref{eq:nu}), for fixed $\varepsilon>0$ we have $\int_{0}^{\infty}t^{-2}\nu_{\varepsilon}(t)\,dt<\infty$. Finally, the claim is a consequence of~\eqref{eq:e1_lem} and the assumption $\lim_{n \to \infty}\kappa(n)=0$.
\end{proof}

\begin{lemma}\label{lem:e2}
Let $p \geq 2$ and $q>1$ be such that $\frac{1}{p}+\frac{1}{q}=1$. Assume that $g \in \mathcal{S}(\mathbb{R})$ and $\varepsilon>0$. For $n \in \mathbb{N}$ we set
\begin{equation}\label{eq:kappa_n_remind}
    \kappa(n)=\left(\frac{1}{n}\max(1,w(B(0,2n)))^{-1}\right)^{1/q}.
\end{equation}
Then we have
\begin{equation}
    \limsup_{n \to \infty}|e_2(n,\varepsilon,\kappa)|=0,
\end{equation}
where $e_2(n,\varepsilon,\kappa)$ is defined in~\eqref{eq:e2}.
\end{lemma}

\begin{proof}
We split $e_2(n,\varepsilon,\kappa)$ as follows:
\begin{align*}
    |e_2(n,\varepsilon,\kappa)| &\leq \frac{16k}{q}n^{-1}\int_{\mathbb{R}}|(\partial_x\Phi)(x/n)|x^{-1}\left(\int_0^{\infty}\nu_{\varepsilon}(t)|P_tf(x)|^{q}\,dt\right)\,dw(x)\\&+\frac{16k}{q}n^{-1}\kappa(n)^{q}\int_{\mathbb{R}}|(\partial_x\Phi)(x/n)|x^{-1}\left(\int_0^{\infty}\nu_{\varepsilon}(t)\,dt\right)\,dw(x)=:e_{2,1}(n,\varepsilon,\kappa)+e_{2,2}(n,\varepsilon,\kappa).
\end{align*}
We will estimate $e_{2,1}(n,\varepsilon,\kappa)$ and $e_{2,2}(n,\varepsilon,\kappa)$ separately.
\\
\noindent
\textbf{Estimate of $e_{2,1}(n,\varepsilon,\kappa)$.} In this case, we use the same argument as in Lemma~\ref{lem:e1}. Recall that $\Phi \in C^{\infty}_c(\mathbb{R})$. Therefore, there is $C>0$ such that
\begin{equation}\label{eq:Phi_one}
    |(\partial_x\Phi)(x/n)| \leq C \text{ for all }x \in \mathbb{R}.
\end{equation}
Moreover, $\supp \Phi \subseteq B(0,2)$ and $\Phi \equiv 1$ on $B(0,1)$, so
\begin{equation}\label{eq:Phi_one_zero}
    (\partial_x \Phi)(x/n) = 0 \text{ for all }x \in B(0,n).
\end{equation}
Consequently,
\begin{equation}
    e_{2,1}(n,\varepsilon,\kappa) \leq \frac{16k}{q}Cn^{-1}\int_{B(0,2n) \setminus B(0,n)}x^{-1}\left(\int_{0}^{\infty}\nu_{\varepsilon}(t)|P_tg(x)|^q\,dt\right)\,dw(x).
\end{equation}
Next, since $|x^{-1}| \leq n^{-1}$ for all $x \not\in B(0,n)$ and, by Lemma~\ref{lem:basic}, $\{P_t\}_{t \geq 0}$ are uniformly bounded on $L^q(dw)$, we get
\begin{align*}
    &\frac{16k}{q}Cn^{-1}\int_{B(0,2n) \setminus B(0,n)}x^{-1}\left(\int_{0}^{\infty}\nu_{\varepsilon}(t)|P_tg(x)|^q\,dt\right)\,dw(x) \leq C'n^{-2}\|f\|_{L^q(dw)}^q\int_{0}^{\infty}\nu_{\varepsilon}(t)\,dt.
\end{align*}
Now, $\limsup_{n \to \infty}e_{2,1}(n,\varepsilon,\kappa)=0$ follows by the fact that $\|f\|_{L^q(dw)}<\infty$, for fixed $\varepsilon>0$ we have $\int_0^{\infty}\nu_{\varepsilon}(t)\,dt<\infty$, and $\lim_{n \to \infty}C'n^{-2}=0$.
\\
\noindent
\textbf{Estimate of $e_{2,2}(n,\varepsilon,\kappa)$.} In this case, we utilize~\eqref{eq:kappa_n_remind}. By~\eqref{eq:Phi_one} and~\eqref{eq:Phi_one_zero} we get
\begin{align*}
    e_{2,2}(n,\varepsilon,\kappa) \leq \frac{16k}{q}Cn^{-1}\kappa(n)^{q}\int_{B(0,2n) \setminus B(0,n)}x^{-1}\left(\int_{0}^{\infty}\nu_{\varepsilon}(t)\,dt\right)\,dw(x). 
\end{align*}
Then, by the fact that $|x^{-1}| \leq n^{-1}$ for $x \not\in B(0,n)$ and by~\eqref{eq:kappa_n_remind},
\begin{align*}
    &\frac{16k}{q}Cn^{-1}\kappa(n)^{q}\int_{B(0,2n) \setminus B(0,n)}x^{-1}\left(\int_{0}^{\infty}\nu_{\varepsilon}(t)\,dt\right)\,dw(x)\\&\leq \frac{16k}{q}Cn^{-3}\frac{1}{w(B(0,2n))}\left(\int_{B(0,2n) \setminus B(0,n)}\,dw(x)\right)\left(\int_{0}^{\infty}\nu_{\varepsilon}(t)\,dt\right)\\&\leq \frac{16k}{q}\frac{Cw(B(0,2n) \setminus B(0,n))}{n^3w(B(0,2n))}\left(\int_{0}^{\infty}\nu_{\varepsilon}(t)\,dt\right).
\end{align*}
Finally, the claim is a consequence of the fact that for fixed $\varepsilon>0$ we have $\int_0^{\infty}\nu_{\varepsilon}(t)\,dt<\infty$.
\end{proof}

As a direct consequence of Lemmas~\ref{lem:part2_modified},~\ref{lem:e1}, and~\ref{lem:e2}, we obtain the following corollary.

\begin{corollary}\label{coro:mod_final}
Assume that $p \geq 2$, $f,g \in \mathcal{S}(\mathbb{R})$, and $g$ is odd. Let $\kappa:\mathbb{N} \to (0,1]$ be defined in~\eqref{eq:kappa_n}. Then we have
\begin{equation}
    \int_0^{\infty}\int_{\mathbb{R}}|t\partial_tP_tf(x)||TP_tg(x)|\,dw(x)\,dt \leq \frac{8}{\gamma}\liminf_{\varepsilon \to 0^{+}}\liminf_{n \to \infty}I(n,\varepsilon,\kappa).
\end{equation}
\end{corollary}

\begin{proof}[Proof of Theorem~\ref{teo:main2}]
Assume first that $p \geq 2$. Take $f \in L^p(dw)$. Thanks to Theorem~\ref{teo:Amri} and the fact that $\mathcal{S}(\mathbb{R})$ is dense in $L^p(dw)$, without loss of generality we can assume $f \in \mathcal{S}(\mathbb{R})$. Let $\kappa:\mathbb{N} \to (0,1]$ be defined by~\eqref{eq:kappa_n} and let $q$ be such that $\frac{1}{p}+\frac{1}{q}=1$. By Proposition~\ref{propo:dual_1_one} we have
\begin{align*}
    \|\mathcal{H}f\|_{L^p(dw)}&=\sup_{g \in \mathcal{S}(\mathbb{R}),\; \|g\|_{L^q(dw)}=1}\left|\int_{\mathbb{R}}\mathcal{H}f(x)g(x)\,dw(x)\right|\\&=4\sup_{g \in \mathcal{S}(\mathbb{R}),\;\|g\|_{L^q(dw)}=1}\left|\int_{\mathbb{R}}\int_0^{\infty}t\partial_tP_tf(x)TP_tg(x)\,dt\,dw(x)\right|.
\end{align*}
Then we split $g$ into even and odd parts $g_{1}$ and $g_{2}$ respectively, so
\begin{equation}\label{eq:split_even_odd}
\begin{split}
    \|\mathcal{H}f\|_{L^p(dw)} &\leq 4\sup_{g \in \mathcal{S}(\mathbb{R}),\;\|g\|_{L^q(dw)}=1}\int_0^{\infty}\int_{\mathbb{R}}\left|t\partial_tP_tf(x)TP_tg_{1}(x)\right|\,dw(x)\,dt\\&+4\sup_{g \in \mathcal{S}(\mathbb{R}),\;\|g\|_{L^q(dw)}=1}\int_0^{\infty}\int_{\mathbb{R}}\left|t\partial_tP_tf(x)TP_tg_{2}(x)\right|\,dw(x)\,dt.
\end{split}
\end{equation}
For $j \in \{1,2\}$, $n \in \mathbb{N}$, and $\varepsilon>0$ let us denote
\begin{equation*}
    I_j(n,\varepsilon,\kappa):=\int_{\mathbb{R}}\Phi(x/n)\int_0^{\infty}\nu_{\varepsilon}(t)(\partial_t^2+\Delta_k)(b_{\kappa(n)}^{\{j\}})(\mathbf{x},t)\,dt\,dw(x),
\end{equation*}
where
\begin{align*}
    b_{\kappa(n)}^{\{j\}}(x,t)=B_{\kappa(n)}(P_tf(x),P_tg_j(x)).
\end{align*}
By Lemma~\ref{lem:part2_even} and Corollary~\ref{coro:part_3} we have
\begin{equation}\label{eq:g1}
\begin{split}
    &4\sup_{g \in \mathcal{S}(\mathbb{R}),\;\|g\|_{L^q(dw)}=1}\left|\int_{\mathbb{R}}\int_0^{\infty}t\partial_tP_tf(x)TP_tg_{1}(x)\,dt\,dw(x)\right| \leq \frac{8}{\gamma}\liminf_{\varepsilon \to 0^{+}}\liminf_{n \to \infty}I_{1}(n,\varepsilon,\kappa)\\&\leq \frac{24(\gamma+1)}{\gamma}(\|f\|_{L^p(dw)}^p+\|g_{1}\|_{L^q(dw)}^q).
\end{split}
\end{equation}
Then, by Lemma~\ref{lem:part2_modified}, Corollary~\ref{coro:mod_final}, and Corollary~\ref{coro:part_3},
\begin{equation}\label{eq:g2}
    \begin{split}
        &4\sup_{g \in \mathcal{S}(\mathbb{R}),\;\|g\|_{L^q(dw)}=1}\left|\int_{\mathbb{R}}\int_0^{\infty}t\partial_tP_tf(x)TP_tg_{2}(x)\,dt\,dw(x)\right| \leq \frac{32}{\gamma}\liminf_{\varepsilon \to 0^{+}}\liminf_{n \to \infty}I_{2}(n,\varepsilon,\kappa)\\&\leq \frac{96(\gamma+1)}{\gamma}(\|f\|_{L^p(dw)}^p+\|g_{2}\|_{L^q(dw)}^q).
    \end{split}
\end{equation}
By the triangle inequality and~\eqref{eq:integral_G_invariance}, for $j \in \{1,2\}$ we have $\|g_j\|_{L^q(dw)} \leq \|g\|_{L^q(dw)}$. Therefore,
by~\eqref{eq:split_even_odd},~\eqref{eq:g1}, and~\eqref{eq:g2},
\begin{align*}
    \|\mathcal{H}f\|_{L^p(dw)} \leq \frac{240(\gamma+1)}{\gamma}\left(\|f\|_{L^p(dw)}^p+\|g\|_{L^q(dw)}^q\right).
\end{align*}
Finally, applying the same polarization arguments as in the proof of Theorem~\ref{teo:main}, we obtain the claim. In case $1<p<2$, one can use the duality argument. 
\end{proof}
\appendix

\section{Proof of\texorpdfstring{~\eqref{eq:bell_2}}{()}}\label{appendix}
We will consider first the function $B(\eta,\zeta)$ defined in~\eqref{eq:B}.

\begin{proposition}
The function $(\eta,\zeta) \longmapsto B(\eta,\zeta)$ is $C^2$ on the set $\mathbb{R}^{N_1} \times \mathbb{R}^{N_2} \setminus \Upsilon$, where
\begin{equation}
    \Upsilon=\{(\eta,\zeta) \in \mathbb{R}^{N_1} \times \mathbb{R}^{N_2}\;:\; \|\eta\|^p=\|\zeta\|^q \text{or }\zeta=\mathbf{0}\}.
\end{equation}
Moreover, for all $(\eta,\zeta) \in \mathbb{R}^{N_1} \times \mathbb{R}^{N_2} \setminus \Upsilon$ and $\omega=(\omega_1,\omega_2) \in \mathbb{R}^{N_1} \times \mathbb{R}^{N_2}$ we have
\begin{equation}\label{eq:HESS}
    \langle \Hess B(\eta,\zeta)\omega,\omega\rangle \geq \frac{\gamma}{2}\left(\tau(\eta,\zeta)\|\omega_1\|^2+\tau(\eta,\zeta)^{-1}\|\omega_2\|^2\right),
\end{equation}
where
\begin{equation}
    \tau(\eta,\zeta)=\|\zeta\|^{2-q}.
\end{equation}
\end{proposition} 

\begin{proof}
We repeat the proof of~\cite[Proposition 6.2]{Mauceri_arxiv}. The regularity properties of $B$ follows directly by the definition of $B$, so we will prove just~\eqref{eq:HESS}. First, we observe that $\langle \Hess B(\eta,\zeta)\omega,\omega\rangle$ is the sum of three summands as follows:
\begin{align*}
    \langle \Hess B(\eta,\zeta)\omega,\omega\rangle&=\sum_{i,j=1}^{N_1}\partial_{\eta_i}\partial_{\eta_j}B(\eta,\zeta)(\omega_1)_i(\omega_1)_j+2\sum_{i=1}^{N_1}\sum_{j=1}^{N_2}\partial_{\eta_i}\partial_{\zeta_j}B(\eta,\zeta)(\omega_1)_i(\omega_2)_j\\&+\sum_{i,j=1}^{N_2}\partial_{\zeta_i}\partial_{\zeta_j}B(\eta,\zeta)(\omega_2)_i(\omega_2)_j=:B_1+B_2+B_3.
\end{align*}
We will estimate $\langle \Hess B(\eta,\zeta)\omega,\omega\rangle$ in two regions:
\begin{align*}
    R_1=\{(\eta,\zeta) \in \mathbb{R}^{N_1} \times \mathbb{R}^{N_2}\;:\; \|\eta\|^p<\|\zeta\|^q,\; \zeta \neq \mathbf{0}\}
\end{align*}
and
\begin{align*}
    R_2=\{(\eta,\zeta) \in \mathbb{R}^{N_1} \times \mathbb{R}^{N_2}\;:\; \|\eta\|^p>\|\zeta\|^q,\; \zeta \neq \mathbf{0}\}.
\end{align*}

\textbf{Estimate in $R_1$.} In this case we have
\begin{align*}
    B(\eta,\zeta)=\frac{1}{2}(\|\eta\|^p+\|\zeta\|^q+\gamma \|\eta\|^2\|\zeta\|^{2-q}).
\end{align*}
Hence, we calculate
\begin{align*}
    \partial_{\eta_i}\partial_{\eta_j}B(\eta,\zeta)=\begin{cases}
        \frac{p(p-1)}{2}\|\eta\|^{p-4}\eta_i\eta_j & \text{ if }i \neq j,\\
        \frac{p(p-1)}{2}\|\eta\|^{p-4}\eta_i\eta_j+\frac{p}{2}\|\eta\|^{p-2}+\gamma \|\zeta\|^{2-q} & \text{ if }i = j,
    \end{cases}
\end{align*}
\begin{align*}
    \partial_{\eta_i}\partial_{\zeta_j}B(\eta,\zeta)=\gamma(2-q)\|\zeta\|^{-q}\eta_i\zeta_j,
\end{align*}
\begin{align*}
    \partial_{\zeta_i}\partial_{\zeta_j}B(\eta,\zeta)=\begin{cases}
        \frac{q(q-2)}{2}\|\zeta\|^{q-4}\zeta_i\zeta_j+\frac{\gamma}{2}q(q-2)\|\eta\|^2\|\zeta\|^{-q-2} & \text{ if }i \neq j,\\
        \frac{q(q-2)}{2}\|\zeta\|^{q-4}\zeta_i\zeta_j+\frac{q}{2}\|\zeta\|^{q-2}+\frac{\gamma}{2}q(q-2)\|\eta\|^2\|\zeta\|^{-q-2}+\frac{\gamma(2-q)}{2}\|\eta\|^2\|\zeta\|^{-q} & \text{ if }i=j.
    \end{cases}
\end{align*}
Thus, in $R_1$,
\begin{align*}
    B_1 = \frac{p(p-2)}{2}\|\eta\|^{p-4} \langle \eta,\omega_1\rangle^2+\frac{p}{2}\|\eta\|^{p-2}\|\omega_1\|^2+\gamma \|\zeta\|^{2-q}\|\omega_1\|^2 \geq \gamma \|\zeta\|^{2-q}\|\omega_1\|^2. 
\end{align*}
Next, note that the condition $\|\eta\|^p<\|\zeta\|^q$ implies $\|\eta\|\|\zeta\|^{1-q}<1$. Therefore,
\begin{align*}
    B_2 &= 2\gamma (2-q)\|\zeta\|^{-q}\langle \eta,\omega_1\rangle \langle \zeta,\omega_2\rangle \\&\geq -2\gamma (2-q)\|\zeta\|^{-q}\|\omega_1\|\|\omega_2\|\|\eta\|\|\zeta\|\\&\geq -\gamma \left(\frac{\|\zeta\|^{2-q}}{2}\|\omega_1\|^2+2\|\zeta\|^{q-2}\|\omega_2\|^2\right).
\end{align*}
In order to estimate $B_3$, note that $\|\eta\|\|\zeta\|^{1-q}<1$ implies $\|\eta\|^2\|\zeta\|^{-q} <\|\zeta\|^{q-2}$. Recall that $\gamma=\frac{q(q-1)}{8}$. Consequently,
\begin{align*}
    B_3 &\geq \frac{q}{2}\|\zeta\|^{q-2}\left((q-2)\|\zeta\|^2\langle \zeta,\omega_2\rangle^2+\|\omega_2\|^2\right)+\frac{\gamma}{2}(2-q)\|\eta\|^2\|\zeta\|^{-q} \left(-q\|\zeta\|^{-2}\langle \zeta,\omega_2\rangle^2+\|\omega_2\|^2\right)\\& \geq\frac{\gamma}{2} \left(8\|\zeta\|^{q-2}+(2-q)(1-q)\|\eta\|^2\|\zeta\|^{-q}\right)\|\omega_2\|^2 \\&\geq \frac{\gamma}{2} \left(8+(2-q)(1-q)\right)\|\zeta\|^{q-2}\|\omega_2\|^2.
\end{align*}
Combining the estimates for $B_1$, $B_2$, and $B_3$ we get
\begin{align*}
    \langle \Hess B(\eta,\zeta)\omega,\omega\rangle &\geq \frac{\gamma}{2}(\|\zeta\|^{2-q}\|\omega_1\|^{2}+(q^2-3q+6)\|\zeta\|^{q-2}\|\omega_2\|^2)\\&\geq \frac{\gamma}{2}(\|\zeta\|^{2-q}\|\omega_1\|^{2}+\|\zeta\|^{q-2}\|\omega_2\|^2),
\end{align*}
so~\eqref{eq:HESS} follows with $\tau(\eta,\zeta)=\|\zeta\|^{2-q}$.

\textbf{Estimate in $R_2$.} In this case we have
\begin{align*}
    B(\eta,\zeta)=\frac{1}{2}\left(\|\eta\|^p+\|\zeta\|^q+\gamma\left(\frac{2}{p}\|\eta\|^p+(\frac{2}{q}-1)\|\zeta\|^q\right)\right),
\end{align*}
so the second derivatives are
\begin{align*}
    \partial_{\eta_i}\partial_{\eta_j}B(\eta,\zeta)=\begin{cases}
        \frac{1}{2}(p+2\gamma)(p-2)\|\eta\|^{p-4}\eta_i\eta_j & \text{ if }i \neq j,\\
        \frac{1}{2}(p+2\gamma)(p-2)\|\eta\|^{p-4}\eta_i\eta_j+\frac{1}{2}(p+2\gamma)\|\eta\|^{p-2} & \text{ if }i=j,
    \end{cases}
\end{align*}
\begin{align*}
    \partial_{\eta_i}\partial_{\zeta_j}B(\eta,\zeta)=0,
\end{align*}
\begin{align*}
    \partial_{\zeta_i}\partial_{\zeta_j}B(\eta,\zeta)=\begin{cases}
        \frac{1}{2}(q+\gamma(2-q))(q-2)\|\zeta\|^{q-4}\zeta_i\zeta_j & \text{ if }i \neq j,\\
        \frac{1}{2}(q+\gamma(2-q))(q-2)\|\zeta\|^{q-4}\zeta_i\zeta_j+\frac{1}{2}(q+\gamma(2-q))\|\zeta\|^{q-2} & \text{ if }i=j.
    \end{cases}
\end{align*}
Hence,
\begin{align*}
    \langle \Hess B(\eta,\zeta)\omega,\omega\rangle &\geq  \frac{1}{2}(p+2\gamma)(p-2)\|\eta\|^{p-4}\langle \eta,\omega_1\rangle^2+\frac{1}{2}(p+2\gamma)\|\eta\|^{p-2}\|\omega_1\|^2\\&+  \frac{1}{2}(q+\gamma(2-q))(q-2)\|\zeta\|^{q-4}\langle \zeta,\omega_2 \rangle^2+\frac{1}{2}(q+\gamma(2-q))\|\zeta\|^{q-2}\|\omega_2\|^2. 
\end{align*}
Recall that $\gamma=\frac{q(q-1)}{8}$. Hence, we have $p+2\gamma \geq 1$ and $q+\gamma(2-q) \geq 1$. Therefore,
\begin{align*}
    \langle \Hess B(\eta,\zeta)\omega,\omega\rangle \geq  \frac{1}{2}((p-1)\|\eta\|^{p-2}\|\omega_1\|^2+(q-1)\|\zeta\|^{q-2}\|\omega_2\|^2).
\end{align*}
Note that $\|\eta\|^p>\|\zeta\|^q$ implies
\begin{align*}
    \|\eta\|^{p-2}=\|\eta\|^{p(p-2)/p}>\|\zeta\|^{q(p-2)/p}=\|\zeta\|^{2-q},
\end{align*}
so
\begin{align*}
    \langle \Hess B(\eta,\zeta)\omega,\omega\rangle \geq  \frac{1}{2}((p-1)\|\zeta\|^{2-q}\|\omega_1\|^2+(q-1)\|\zeta\|^{q-2}\|\omega_2\|^2).
\end{align*}
Note that $p-1 \geq 1$ and $q-1 \geq \gamma$. Therefore,
\begin{align*}
    \langle \Hess B(\eta,\zeta)\omega,\omega\rangle \geq  \frac{\gamma}{2}(\|\zeta\|^{2-q}\|\omega_1\|^2+\|\zeta\|^{q-2}\|\omega_2\|^2)
\end{align*}
and we can take $\tau(\eta,\zeta)=\|\zeta\|^{2-q}$ in~\eqref{eq:HESS}.
\end{proof}

\begin{proof}[Proof of~\eqref{eq:bell_2}]
We repeat the argument from~\cite[Theorem 4]{Sch} and~\cite[Proposition 6.3]{Mauceri_arxiv}. It follows by the formulas for the second derivatives of $B$ which are given above that they are $C^2$ on $\mathbb{R}^{N_1} \times \mathbb{R}^{N_2} \setminus \Upsilon$ and they are locally integrable. Moreover, $B$ is $C^1$ on $\mathbb{R}^{N_1} \times \mathbb{R}^{N_2}$. That means that the distributional derivatives of $B$ exist and they coincide with the usual ones on $\mathbb{R}^{N_1} \times \mathbb{R}^{N_2} \setminus \Upsilon$. Hence, we get the identity
\begin{align*}
    \langle \Hess B_{\kappa}(\eta,\zeta)\omega,\omega\rangle=\int_{\mathbb{R}^{N_1} \times \mathbb{R}^{N_2} \setminus \Upsilon}\phi_{\kappa}(\eta-\eta_1,\zeta-\zeta_1)\langle \Hess B(\eta_1,\zeta_1)\omega,\omega\rangle\,d\eta_1\,d\zeta_1
\end{align*}
for all $(\eta,\zeta) \in \mathbb{R}^{N_1} \times \mathbb{R}^{N_2}$ and $\omega \in \mathbb{R}^{N_1+N_2}$. By~\eqref{eq:HESS} we obtain
\begin{align*}
    &\int_{\mathbb{R}^{N_1} \times \mathbb{R}^{N_2} \setminus \Upsilon}\phi_{\kappa}(\eta-\eta_1,\zeta-\zeta_1)\langle \Hess B(\eta_1,\zeta_1)\omega,\omega\rangle\,d\eta_1\,d\zeta_1 \\&\geq \frac{\gamma}{2}\int_{\mathbb{R}^{N_1} \times \mathbb{R}^{N_2}}\phi_{\kappa}(\eta-\eta_1,\zeta-\zeta_1) \tau(\eta_1,\zeta_1)\|\omega_1\|^{2}d\eta_1\,d\zeta_1 \\&+\frac{\gamma}{2}\int_{\mathbb{R}^{N_1} \times \mathbb{R}^{N_2} }\phi_{\kappa}(\eta-\eta_1,\zeta-\zeta_1) \tau^{-1}(\eta_1,\zeta_1)\|\omega_2\|^{2}d\eta_1\,d\zeta_1,
\end{align*}
so~\eqref{eq:bell_2} is proved.
\end{proof}

\bibliographystyle{amsplain}
\bibliography{bib}

\end{document}